\documentclass[a4paper,11pt,leqno,english]{smfart_ahmed}
\usepackage{aeguill}
\usepackage{enumerate}
\usepackage{amssymb,amsmath,latexsym,amsthm}
\usepackage[T1]{fontenc}
\usepackage{geometry}
\usepackage{url}
\usepackage[frenchb, main=english]{babel}
\usepackage[utf8]{inputenc}
\usepackage{mathrsfs}
\usepackage{xcolor}
\usepackage{comment}
\definecolor{violet}{rgb}{0.0,0.2,0.7}
\definecolor{rouge2}{rgb}{0.8,0.0,0.2}
\usepackage{tikz}
\usepackage{empheq}
\usepackage{tikz-cd}
\usetikzlibrary{matrix,arrows,decorations.pathmorphing}
\usepackage{hyperref}
\usepackage{mathpazo}
\hypersetup{
    bookmarks=true,         
    unicode=false,          
    pdftoolbar=true,        
    pdfmenubar=true,        
    pdffitwindow=false,     
    pdfstartview={FitH},    
    pdftitle={},    
    pdfauthor={},     
    colorlinks=true,       
   linkcolor=rouge2,          
    citecolor=violet,        
    filecolor=black,      
    urlcolor=cyan}           
\usepackage{enumitem}
\usepackage{appendix}

\newcommand{\wh}{\widehat}
\newcommand{\Imm}{\mathrm{Im} \,}

\renewcommand{\d}{\partial}

\newcommand{\ddbar}{\partial\bar{\partial}}

\newcommand{\cX}{\mathcal{X}}

\newcommand{\cC}{\mathcal{C}}
\newcommand{\cG}{\mathcal{G}}

\renewcommand{\O}{\mathcal{O}}

\newcommand{\ep}{\varepsilon}
\renewcommand{\epsilon}{\varepsilon}

\renewcommand{\ker}{\mathrm{Ker} \,}

\newcommand{\cF}{\mathcal{F}}

\newcommand{\ol}{\overline}

\renewcommand{\ge}{\geqslant}
\renewcommand{\le}{\leqslant}
\renewcommand{\leq}{\leqslant}
\renewcommand{\geq}{\geqslant}

\newcommand{\om}{\omega}

\newcommand{\ddc}{dd^c}

\renewcommand{\tt}{\theta_t}

\newcommand{\uu}{\mathbf{u}}
\newcommand{\vv}{\mathbf{v}}
\newcommand{\dbar}{\bar \partial}

\newtheorem{thm}{Theorem}[section]
\newtheorem{lemme}[thm]{Lemma}
\newtheorem{proposition}[thm]{Proposition}
\newtheorem{claim}[thm]{Claim}
\newtheorem{question}[thm]{Question}

\newtheorem{defn}[thm]{Definition}
\newtheorem{cor}[thm]{Corollary}

\theoremstyle{remark}
\newtheorem{setup}[thm]{Set-up}
\newtheorem{notation}[thm]{Notation}
\newtheorem{remark}[thm]{Remark}

\theoremstyle{plain}
\newtheorem{bigthm}{Theorem}

 \numberwithin{equation}{section}

\title[Curvature formula in a singular setting]{Curvature formula for direct images of twisted relative canonical bundles endowed with a singular metric}

\author{Junyan Cao}
\address{Laboratoire de Mathématiques J.A. Dieudonné, UMR 7351 CNRS, Université Côte d'Azur, Parc Valrose, 06108 Nice Cedex 02, France \qquad \qquad \qquad \qquad \qquad \qquad \qquad \qquad \qquad}

\email{junyan.cao@unice.fr}

\author{Henri Guenancia}
\address{Institut de Mathématiques de Toulouse; UMR 5219, Université de Toulouse; CNRS, UPS, 118 route de Narbonne, F-31062 Toulouse Cedex 9, France}
\email{henri.guenancia@math.cnrs.fr}

\author{Mihai P\u{a}un}
\address{Institut für Mathematik, Universität Bayreuth, 95440 Bayreuth, Germany}
\email{mihai.paun@uni-bayreuth.de}

\begin{document}
\begin{abstract}
In this note, we obtain various formulas for the curvature of the $L^2$ metric on the direct image of the relative canonical bundle twisted by a holomorphic line bundle endowed with a positively curved metric with analytic singularities, generalizing some of Berndtsson's seminal results in the smooth case.  When the twist is assumed to be relatively big, we further provide a very explicit lower bound for the curvature of the $L^2$ metric. 
\end{abstract}
\maketitle

\tableofcontents

\section{Introduction}

Let $p:\cX\to D$ be a smooth, proper fibration from a $(n+1)$-dimensional Kähler manifold $\cX$ onto the unit disk $D\subset \mathbb C$, and let $(L,h_L)$ be a holomorphic line bundle endowed with a possibly singular hermitian metric $h_L$ assumed to be positively curved (i.e. when $i\Theta_{h_L} (L) \geq 0$ in the sense of currents). Then, the positivity properties of the direct image sheaves 
$$\cF:= p_\star\left((K_{\cX/D} +L)\otimes \mathcal I(h_L)\right)$$
endowed with the $L^2$ metric $h_\cF$ are well-known, cf e.g. \cite{Bo09, PT, Paun18, DNWZ20} among many others. Moreover, when $h_L$ is {\it smooth}, we have at hand explicit formulas obtained by Berndtsson \cite{Bo09, Bo11} that compute the curvature of the $L^2$ metric on the direct image sheaf above. 

\medskip 

\noindent In this article we are aiming at the generalisation of Berndtsson's curvature formulas in case where the metric $h_L$ has relatively simple singularities, e.g. \textit{analytic singularities}. This is partly motivated by the need to have an interpretation of the "flat directions" in the curvature of $(\cF, h_\cF)$ in this context. Our main result in this direction states as follows.

\begin{bigthm}
\label{thmA}
Let $p:\cX\to D$ and $(L,h_L)\to \cX$ as above, and let $u\in H^0(D,\cF)$. We assume that 
\begin{enumerate}[label=$\bullet$]
\item The metric $h_L$ has analytic singularities and $i\Theta_{h_L}(L)\ge 0$ in the sense of currents. 
\item The section $u$ is flat with respect to $h_\cF$. 
\end{enumerate} 
Set $E :=\{h_L =\infty\}$. Then, there exists a continuous $L^2$-integrable representative $\uu$ of $u$ defined on the restriction $\cX^\star \setminus E$ of the family $p$ to some punctured disk $D^\star$ such that  
 \[ \frac{\dbar \uu}{dt} \Big|_{X_t\setminus E} =0 \]
for any $t\in D^\star$ and 
\begin{equation}
D'\uu= 0,  \qquad  \quad \Theta_{h_L} (L) \wedge \uu= 0 
\end{equation}
on $\cX^\star \setminus E$. 
Here $\cX^\star:=p^{-1}(D^\star)$ and $\uu$ is $L^2$ with respect to $h_L$ and a Poincaré type metric cf. Section \ref{sec2}.
\end{bigthm}

\noindent By "punctured disk" in the previous statement we mean that $D\setminus D^\star$ is a discrete set, possibly empty. By $L^2$, we mean locally $L^2$ with respect to the base $D^\star$. 

\medskip

\noindent The result we are next mentioning concerns the case of a twisting line bundle $L$ which is $p$-big. It is then expected that the strict positivity of $(L, h_L)$ is inducing stronger positivity properties of the curvature of
the direct image than in the general case of a semi-positively curved $L$. 
This is confirmed by the following statement, which is a version of \cite[Thm 1.2]{Bo11}. 
 
 \begin{bigthm}
\label{thmB}
Let $p:\cX\to D$ be a smooth projective fibration and let $(L,h_L)\to \cX$ be a line bundle such that \begin{enumerate}[label=$\bullet$]
\item $h_L$ has analytic singularities and $i\Theta_{h_L}(L)\ge 0$ in the sense of currents.
\item For any $t\in D$, the absolutely continuous part $\om_L:=(i\Theta_{h_L}(L))_{\rm ac}$ satisfies $\int_{X_t} \om_L^n>0$. 
\end{enumerate}
Then there exists a punctured disk $D^\star\subset D$ such that 
for any $u \in H^0 (D, \cF)$ we have the following inequality
\begin{equation}\label{c1 intro}
\langle\Theta_{h_\cF}(\cF)u, u\rangle_t\geq c_n\int_{X_t} c(\om_L)u\wedge \ol u e^{-\phi_L}
\end{equation}
for any $t\in D^\star $. 
\end{bigthm}

\noindent In the statement above we identify $\Theta_{h_\cF}(\cF)$ with an endomorphism of $\cF$ 
by "dividing" with $idt\wedge d\ol t$. Moreover, $c_n=(-1)^{\frac {n^2}2}$ is the usual unimodular constant. We denote by $c (\om_L) := \frac{\om_L^{n+1}}{\om_L^{n} \wedge  idt \wedge d \overline{t}}$ the geodesic curvature associated to $\om_L$, cf. Definition~\ref{geod curv} for a precise definition in the degenerate case.

Actually we can provide some details about the punctured disk $D^\star$ in Theorem~\ref{thmB}.
Under the hypothesis of this result, it turns out that the $L^2$ metric $h_\cF$
is smooth in a complement of a discrete subset of $D$.
We will show that the formula \eqref{c1 intro} is valid for points $t\in D$ in the neighborhood of which the metric $h_\cF$ is smooth, and such that $\displaystyle \cF_t= H^0\left(X_t, (K_{X_t}+ L)\otimes \mathcal I(h_L|_{X_t})\right)$, cf Remark~\ref{rem Dstar}.
\medskip

$\bullet$ {\bf Strategy of the proof}

\noindent
 Roughly speaking, the idea of the proof of Theorem \ref{thmA} and \ref{thmB} respectively is as follows: we endow the complement $\cX\setminus E$ with a complete metric of Poincar\'e type and proceed by taking advantage of what is known in the compact case, combined with the existence of families of cut-off functions specific to the complete setting. There are however quite a few difficulties along the way. Probably the most severe stems from the Hodge decomposition in the complete case: the image of the usual operators $\dbar$ and $\dbar^\star$ may not be closed. We show in Section~\ref{hodge} that at least in bi-degree $(n,1)$ this is the case, cf.~Theorem~\ref{Hodge},
as consequence of the fact that the background metric has Poincar\'e singularities.

\noindent In order to construct the form $\uu$ in Theorem~\ref{thmA}, we start with a representative of $u$ given by the contraction with the lifting $V$ of $\displaystyle \frac{\partial}{\partial t}$ with respect to a Poincar\'e metric $\om_E$. It turns out that this specific representative has all the desired properties needed to fit into the $L^2$-theory.
Then we "correct" it: this is possible by the flatness hypothesis, and it boils down to solving a fiberwise $\dbar^\star$-equation. It is both in the resolution of this equation as well as in the study of the regularity of the resulting solution that Theorem~\ref{Hodge} is used. 
Another important ingredient of the proof is Proposition \ref{curvv}, which gives a general curvature formula for $(\cF,h_\cF)$ when $h_L$ has e.g. analytic singularities. It provides a rather wide generalisation of a result due to Berndtsson. 

\noindent As for Theorem~\ref{thmB}, the starting point is the fact that the positivity properties of $(L,h_L)$ allow us to construct a family of Poincaré metric $(\om_\ep)_{\ep >0}$ on $\cX\setminus E$. Then, the representatives $\uu_\ep$ of $u$ \-- obtained as above as the contraction with the lifting $V_\ep$ of $\displaystyle \frac{\partial}{\partial t}$ with respect to $\om_\ep$ \-- enjoy a special property that allows us to extract the desired  inequality from the general curvature formula from Proposition~\ref{curvv} and a limiting argument when $\ep$ approaches zero. Although the use of this special representative goes back to Berndtsson, several new analytic inputs are required to deal with the present singular situation. \\

$\bullet$ {\bf Organization of the paper}

\begin{enumerate}[label=$\circ$]
\item In Section~\ref{sec1}, we introduce our set-up, notation and main objects of study (the $L^2$ metric on $\cF$, the geodesic curvature).
\item In Section~\ref{sec2}, we review two aspects of Poincaré metrics: first, the integrability properties of representatives $\uu$ of sections $u$ of $\cF$ constructed via such metrics (Lemma~\ref{meta}) and then, we investigate the closedness of the image of the operators $\dbar, \dbar^*$ on a hermitian line bundle with analytic singularities (Theorem~\ref{Hodge}). 
\item In Section~\ref{sec3}, we establish a general curvature formula (Proposition~\ref{curvv}). This allows us to find very special representatives of flat sections of $\cF$ (Theorem~\ref{mainresI}), leading to the proof of Theorem~\ref{thmA}. 
\item In Section~\ref{section 5}, we analyze the relatively big case in the "snc situation" (Theorem~\ref{strikt}), from which we then deduce
Theorem~\ref{thmB} .
\end{enumerate}

$\bullet$ {\bf Acknowledgements}
H.G. has benefited from State aid managed by the ANR under the "PIA" program bearing the reference ANR-11-LABX-0040,
 in connection with the research project HERMETIC.
J.C. thanks the excellent working conditions provided by the IHES during the main part of the preparation of the article. M.P. gratefully acknowledge the support of DFG. 

\noindent \emph{It is our privilege to dedicate this article to our friend and colleague Ahmed~Zeriahi,  with our admiration for his outstanding mathematical achievements and wishing him a very happy and active retirement}!

\section{Set-up and notation}
\label{sec1}

\noindent The set of assumptions we need for our results to hold is the following.


\begin{setup}
\label{altsetup}
Let $p:\cX\to D$ be a smooth, proper fibration from a $(n+1)$-dimensional K\"ahler manifold $\cX$ onto the unit disk $D\subset \mathbb C$, and let $(L,h_L)$ be a holomorphic line bundle endowed with a possibly singular hermitian metric $h_L$.
\smallskip

\noindent We assume that there exists a divisor $E= E_1+\dots + E_N$  whose support is contained in the total space $\cX$ of $p$.
such that the following requirements are fulfilled.
\begin{enumerate}[label={\color{violet}(A.\arabic*)}]
\item \label{A1} For every $t\in D$ the divisor $E+X_t$ has simple normal crossings. Let $\Omega\subset \cX$ be a coordinate subset on $\cX$. We take $(z_1,\dots. z_n, t= z_{n+1})$ a coordinate system on $\Omega$ such that the last one $z_{n+1}$ corresponds to the map $p$ itself and such that $z_1\dots z_p= 0$ is the local equation of $E\cap \Omega$.

\item \label{A2} The metric $h_L$ has {\it generalised} analytic singularities along $E$; i.e. its local weights $\varphi_L$ on $\Omega$ can be written as
\begin{equation*}
\varphi_L\equiv \sum_{i=1}^p a_i\log|z_i|^2- \sum_{I} b_I\log\left(\phi_I(z)-\log \big(\prod_{i\in I}|z_i|^{2k_i}\big)\right) 
\end{equation*}
modulo $\mathcal C^\infty$ functions, where $a_i, b_I$ are positive real numbers, $k_i$ are positive integers and $(\phi_I)_I$ are smooth functions on $\Omega$. The set of indexes in the second sum coincides with the non-empty subsets of $\{1,\dots, p\}$.

\item \label{A3} The Chern curvature of $(L,h_L)$ satisfies
\[i\Theta_{h_L}(L)\ge 0\]
in the sense of currents on $\cX$. 
\end{enumerate}
\medskip  

\noindent We then set
\[\mathcal F:=p_*((K_{\cX/D}+L)\otimes \mathcal I(h_L))\]
and assume that that this vector bundle on $D$ has positive rank.
As a consequence of the previous requirements \ref{A1}-\ref{A3}, we have the following statement.

\begin{lemme}\label{prelemme}
Under the assumptions  \ref{A1}-\ref{A3}, we have
\begin{equation*}
\cF_t= H^0\left(X_t, (K_{X_t}+ L)\otimes \mathcal I(h_L|_{X_t})\right)
\end{equation*}
for every $t\in D$. Moreover, the canonical $L^2$ metric (cf. Notation \ref{notation}) on $\cF$ is non-singular.   
\end{lemme}
\begin{proof}
We first remark that $\cF$ is indeed locally free 
given that it is torsion-free and $D\subset \mathbb C$ is a disk.  

The fibers of $\cF$ are indeed identified with $H^0\left(X_t, (K_{X_t}+ L)\otimes \mathcal I(h_L|_{X_t})\right)$ because of the transversality hypothesis \ref{A1}, combined with the type of singularities we are allowing for $h_L$ in \ref{A2}. The point is that a holomorphic function $f$
defined on the coordinate subset $\Omega$ belongs to $\mathcal I(h_L)$ exactly when the restriction $\displaystyle f|_{\Omega\cap X_t}$ belongs to the ideal
$\mathcal I(h_L|_{X_t})$. On the other hand the Kähler version of Ohsawa-Takegoshi theorem \cite{CaoOT} implies that any element of $H^0\left(X_t, (K_{X_t}+ L)\otimes \mathcal I(h_L|_{X_t})\right)$ extends to $\cX$
--it is at this point that the hypothesis \ref{A3} plays a crucial role.

 Concerning the smoothness of the $L^2$-metric on $\cF$, we can use partitions of unity to reduce to checking that integrals of the form $\int_{\Omega\cap X_t} |f_t|^2e^{-\varphi_L}$ vary smoothly with $t$, where $f_t=f|_{X_t}$ for some $f\in \mathcal I(h_L)|_{\Omega}$ and $\varphi_L$ is given by the expression in  \ref{A2}. Now it is clear there that all derivatives in the $t,\bar t$ variables of $\varphi_L$ are bounded, so that the result follows from general smoothness results for integrals depending on a parameter. 
\end{proof}  
\end{setup}

\medskip

\noindent $\bullet $ {\it A few comments about the conditions \ref{A1}-\ref{A2}.}

\noindent The point we want to make here is that the transversality requirements in  \ref{A1}-\ref{A2} can be obtained starting from a quite general context.

We consider $p:\cX\to D$ a proper fibration from a $(n+1)$-dimensional Kähler manifold $\cX$ onto the unit disk $D\subset \mathbb C$, and let $(L,h_L)$ be a holomorphic line bundle endowed with a possibly singular hermitian metric $h_L$.
We assume that \ref{A3} holds true, and that the singularities of $h_L$ are of the form
\begin{equation}\label{s2}
\varphi_L\equiv \sum_{i=1}^p a_i\log|f_i|^2- \sum_{i=1}^p b_i\log\left(\tau_i- \log |g_i|^2 \right)
\end{equation}
modulo $\mathcal C^\infty$ functions, where $a_i, b_i$ are positive real numbers, $f_i, g_i$ are holomorphic, and $\tau_i$ are smooth.

If  \ref{A1}-\ref{A2} are not satisfied for $p:\cX\to D$, then one can consider a log resolution $\pi:\cX'\to \cX$ of $(\cX, \mathscr I_Z)$ where $Z$ is the singular set of $h_L$. Set $p':=p\circ \pi:\cX'\to D$, $(L',h_{L'}):=(\pi^*L, \pi^*h_L)$, $E$ is the reduced divisor induced by $\pi^{-1}(Z)$. It is immediate that 
$$\pi_*((K_{\cX'/D}+L')\otimes \mathcal I(h_{L'}))=(K_{\cX/D}+L)\otimes \mathcal I(h_L)$$ 
so that, in particular, $\pi_*((K_{\cX'/D}+L')\otimes \mathcal I(h_{L'}))=\mathcal F$. 

The map $p'$ may not be smooth anymore (e.g. some components of $E$ may be irreducible components of fibers of $p'$). Set $D_{\rm reg}$ to be the Zariski open set of regular values of $p'$, $D_1:=D_{\rm reg}\cap D$, $\cX_1:=p'^{-1}(D_1)$, $p_1:=p'|_{\cX_1}$, $(L_1, h_{L_1})$:=$(L',h_{L'})|_{\cX_1}$. Then, the triplet $(p_1,L_1,h_{L_1})$ satisfies the assumptions  \ref{A1}-\ref{A3}.
\medskip

\noindent In conclusion, starting with a map $p$ as above and singular metric $h_L$ as in \eqref{s2}, we can use our results on the family $p_1$ restricted to some punctured disk $D^\star\subset D$. 

\medskip

\noindent $\bullet $ {\it The geodesic curvature in a degenerate setting.}

\noindent Let $p:\cX\to D$ be a smooth, proper fibration where $\cX$ is a Kähler manifold of dimension $n+1$. Let $\omega$ be a closed positive $(1,1)$-current on $\cX$ such that $\om$ is smooth on a non-empty Zariski open subset $\cX^\circ\subset \cX$. Let $\om_\cX$ be a Kähler metric on $\cX$. 

\begin{defn}
\label{geod curv}
The geodesic curvature $c(\om)$ of $\om$ on $\cX^\circ$ is defined by
\[c(\om):=\lim_{\ep \to 0} c(\om+\ep \om_\cX) = \lim_{\ep \to 0} \frac{(\om+\ep\om_\cX)^{n+1}}{(\om+\ep\om_\cX)^{n}\wedge idt\wedge d\bar t}.\]
\end{defn}
A few explanations are in order. 

First, if $\om$ is relatively Kähler on $\cX$, one recovers the usual definition. 

Next, it is easy to observe that $c(\om+\ep\om_\cX)=1/\|dt\|^{2}_{\om+\ep\om_\cX}$. In particular, that non-negative quantity is non-increasing when $\ep$ decreases towards $0$, hence it admits a limit. By the same token, one can see that the limit is independent of the choice of the Kähler metric $\om_\cX$. 

Finally, if $t\in D$ is such that $X_t^\circ:=X_t\cap \cX^\circ$ is dense in $X_t$, then one defines $c(\om)$ on the whole $X_t$ by extending it by zero across $X_t\setminus X_t^\circ$. Note that if the absolutely continuous part of $\om$ satisfies $\om_{\rm ac} \le C \om_{\cX}$ on $X_t$ for some constant $C>0$, then $c(\om)$ is a bounded function on $X_t$ (this follows e.g. from the inequality $c(\om)\le \|\frac{\partial}{\partial t}\|^2_{\om}$ for a set of coordinates $(z_1, \ldots, z_n, z_{n+1}=t)$ such that $p(z)=t$). In particular, the integral $\int_{X_t} c(\om) \om_{\cX}^n$ is finite. 

\medskip

\begin{notation}
\label{notation}
In the Set-up~\ref{altsetup} above:

$\cdot$ We set $\cX^\circ:=\cX\setminus E$, $X_t^\circ:=X_t \cap \cX^\circ$,  $L_t:=L|_{X_t}$, $h_{L_t}:=h_{L}|_{X_t}$. 

$\cdot$ We use interchangeably $h_L$ and $e^{-\phi_L}$; when working in a trivializing chart of $L$, we will denote by $\varphi_L$ the local weight of $h_L$. The $(1,0)$-part of the Chern connection of $(L,h_L)$ over $\cX\setminus E$ is denoted by $D'$. 

$\cdot$ Under assumption \ref{A1}, we will write  $E :=\sum_{i=1}^N  E_i$ for the decomposition of $E$ into its (smooth) irreducible components.  Next, let $s_i$ be a section of $\mathcal O_\cX(E_i)$ that cuts out $E_i$, and let $h_i$ be a smooth hermitian metric on $\mathcal O_\cX(E_i)$. 
In the following, $|s_i|^2$ stands for $|s_i|^2_{h_i}$, and we assume that $|s_i|^2<e^{-1}$.

 $\cdot$ We will interchangeably denote by $\|\cdot \|$ or $h_\cF$ the $L^2$ metric on $\mathcal F$; i.e. if $u\in \mathcal F_t=H^0(X_t, (K_{X_t}+L_t)\otimes \mathcal I(h_{L_t}))$, then $\|u\|^2:=c_n \int_{X_t}u\wedge \bar u e^{-\phi_{L_t}}$ with $c_n=(-1)^{\frac{n^2}{2}}$. Lemma \ref{prelemme} ensures that the $L^2$ metric is {\it smooth} on $D$. We denote by $\nabla$ the $(1,0)$ part of the Chern connection of $(\mathcal F, \|\cdot \|)$ on $D$.
\end{notation}


\section{A few technicalities about  Poincar\'e type metrics} \label{sec2}

\noindent Throughout this section we adopt Set-up~\ref{altsetup}. 
Let $\omega$ be a fixed Kähler metric on $\cX$, and let 
$$\om_E  := \omega + \ddc \Big[ - \sum_{i=1}^N \log \log \frac{1}{|s_i|^2} \Big]  \qquad\text{on } \cX^\circ$$
be a metric with Poincar\'e singularities along $E$.
Thanks to \ref{A1} we infer that $\displaystyle \om_E |_{X_t^\circ}$ is a complete K\"ahler metric on $X^\circ_t$ with Poincar\'e singularities along $E\cap X_t$ for each $t\in D$.
\medskip

\noindent In the next subsections we will be concerned with the following two main themes.

Let $V$ be the horizontal lift of $\displaystyle \frac{\partial}{\partial t}$ with respect to the Poincar\'e-type metric $\om_E$. We estimate the size of its coefficients near the singularity divisor $E$, and then show that
the representatives of direct images constructed by using $V$ have the expected $L^2$ properties allowing us to use them in the computation of the curvature of $\cF$. This is the content of Subsection \ref{flit}.  

In Subsection \ref{hodge} we establish a few important properties of the $L^2$-Hodge decomposition for $(n,1)$-forms with values in $(L, h_L)$, where the background metric is $(X, \om_E)$. The main result here is that the image of $\dbar^\star$ is closed, cf Theorem~\ref{Hodge},  a very useful result \emph{per se} and for the next sections of this paper as well.

\subsection{Estimates of the background metric, lifting}\label{flit}


\bigskip


We choose local coordinates $(z_1,\dots, z_n, z_{n+1}= t)$ on $\cX$ such that $p(z,t)=t$ --as in \ref{A1}--
and in which $\om_E$ is locally given by 
$$\om_E = g_{t\bar t}\,idt \wedge d\bar t+ \sum_{\alpha} g_{\alpha \bar t }\, idz_{\alpha} \wedge d\bar t+ \sum_{\alpha} {g}_{t \bar \alpha} \, idt \wedge d \bar z_{\alpha}+\sum_{\alpha, \beta} g_{\alpha\bar \beta} \,idz_{\alpha} \wedge d\bar z_{\beta} $$
By the estimates in \cite[\S 4.2]{G12} the coefficients of $\om_E$ are can be written as follows
\begin{equation}
\label{coeff poin poin}
g_{j \bar k}=g_{j\bar k}^0+ \frac{\delta_{j}\cdot \delta_{jk}}{|z_{j}|^{2}\log^{2}|z_{j}|^{2}}+\frac{\delta_{j} A_{j}}{z_j \log^{2}|z_j|^{2}}+\frac{\delta_{k} B_k}{\bar z_k \log^{2}|z_k|^{2}}+ \sum^{p}_{\ell=1}\frac{C_{\ell}}{\log |z_{\ell}|^{2}} 
\end{equation}
where $A_j, B_k, C_{\ell},g_{j\bar k}^0$ are smooth functions on $\Omega$. We use the notation $\delta_{j}=\delta_{j\in \{1,\ldots, p\}}$ and $\delta_{ij}$ is the usual Kronecker symbol. 
\smallskip

\noindent In order to present the computations to follow in a reasonably simple way, we introduce for $i=1,\ldots, n$ the 
functions 
\begin{equation}\label{a1} f_i(z):=\begin{cases} w_i:= z_i \log{|z_i|^{2}} &\mbox{if} \quad i\in\{1, \ldots, p\}\\ 1 & \mbox{else}\end{cases}
\end{equation}
or $f_i(z)= \delta_iw_i+ 1-\delta_i$ more concisely. The coefficients of the metric $\om_E$ can be written as 
\begin{equation}\label{a2}
g_{\alpha \bar \beta}= \frac{1}{f_\alpha\bar f_\beta}\Psi_{\alpha\bar \beta}(z'', w, \rho).
\end{equation}
The notations we are using in \eqref{a2} are:
\begin{enumerate}

\item[(a)] $\displaystyle \Psi_{\alpha\bar \beta}$ is a smooth function defined in the neighborhood of $0\in \mathbb C^{n-p}\times  \mathbb C^{p}\times  \mathbb C^{p}$.

\item[(b)] $w:=(w_1, \ldots, w_p)$ (cf. \eqref{a1}) and $z'':= (z_{p+1},\dots , z_n)$.

\item[(c)]  For $i=1,\dots, p$ we introduce $\displaystyle \rho_i:= \frac{1}{\log|z_i|^2}$. 
\end{enumerate}
\smallskip

\noindent Given that $\omega_E$ is a metric, the functions $\Psi_{\alpha\bar\beta}$ are not arbitrary (since the matrix $(\Psi_{\alpha \bar \beta})_{\alpha\bar \beta}$ is definite positive at each point).
We simply want to emphasize in \eqref{a2} the general shape of the coefficients, which will be useful in the statement that follows.

\begin{lemme}\label{coef}
The following estimates hold true:
\begin{enumerate}

\item[(i)] $\displaystyle \det(g)=  \prod_{\alpha=1}^n |f_\alpha(z)|^{-2}\left(1+ \Psi(z'', w, \rho)\right)$,
where the {\rm "1"} inside the parentheses means a strictly positive constant, and $\Psi$ is smooth such that 
$\Psi(0)= 0$.

\item[(ii)] For each pair of indexes $\alpha, \beta$ we have
$$g^{\bar \beta\alpha}= f_\alpha(z)\bar f_\beta(z) \Psi^{\alpha\bar\beta}(z'', w, \rho),$$
where $g^{\ol \beta \alpha}$ are the coefficients of the inverse of $(g_{\alpha\bar\beta})$.
\end{enumerate}
\end{lemme}
\begin{proof}
Both statements above are obtained by a direct calculation, using the expression \eqref{a2} of the coefficients. We skip the straightforward details.
\end{proof}
\medskip

\noindent Thanks to Lemma \ref{coef} it is easy to infer the following useful estimates: for each set of indexes $\alpha, \beta, q, r$ we have  
{\small
\begin{align}
\label{a4} \frac{\partial g^{\ol \beta \alpha}}{\partial z_q}(z) =&   \left(\delta_{\alpha q}\delta_\alpha (1+ \log|z_\alpha|^2)\bar f_\beta(z)+ \delta_{\beta q}\delta_\beta f_\alpha(z)\frac{\ol z_q}{z_q}\right)\Psi_{\alpha\bar\beta}(z'', w, \rho)\\ 
\nonumber & + (1+\delta_q (\log|z_q|^2-1)) f_\alpha(z)\bar f_{ \beta}(z)\Psi_{\alpha\bar\beta}(z'', w, \rho)  + \frac{\delta_q}{z_q\log^2|z_q|^2}f_\alpha(z)\bar f_{ \beta}(z)\Psi_{\alpha\bar\beta}(z'', w, \rho)
\end{align}}
as well as
\begin{align}
\frac{\partial^2 g^{\ol \beta \alpha}}{\partial z_q\partial \ol z_r}(z)  \label{a5}= &  {\, \mathcal O}(1)\frac{\partial}{\partial \ol z_r}\left(\delta_{\alpha q}\delta_\alpha (1+ \log|z_\alpha|^2)\bar f_\beta(z)+ \delta_{\beta q}\delta_\beta f_\alpha(z)\frac{\ol z_q}{z_q}\right)\\
\nonumber & + {\mathcal O}(1)\left(1+ \frac{\delta_r}{\ol z_r\log|z_r|^2}\right)\left(
\delta_{\alpha q}\delta_\alpha (1+ \log|z_\alpha|^2) \bar f_\beta(z)+ \delta_{\beta q}\delta_\beta f_\alpha(z)\frac{\ol z_q}{z_q}\right)\\
\nonumber & + {\mathcal O}(1)\frac{\delta_q\delta_{rq}}{|z_q|^2\log^3|z_q|^2}f_\alpha(z)\bar f_{ \beta}(z)\\ 
\nonumber & + {\mathcal O}(1)\left(1+ \frac{\delta_q}{z_q\log|z_q|^2}\right)\left(1+ \frac{\delta_r}{\ol z_r\log|z_r|^2}\right)f_\alpha(z)\bar f_{ \beta}(z)\\
\nonumber & + {\mathcal O}(1)\left(1+ \frac{\delta_q}{z_q\log|z_q|^2}\right)\left(\delta_{\beta r}\delta_\beta (1+ \log|z_\beta|^2)f_\alpha(z)+ \delta_{\alpha r}\delta_\alpha \bar f_\beta(z)\frac{z_r}{\ol z_r}\right)
\end{align}

\noindent Again, in the relations \eqref{a4}--\eqref{a5} we are using $\Psi_{\alpha\ol \beta}$ and $\mathcal O(1)$ as generic notation, these functions are allowed to change from one line to another, the point is that they are of the same type. The verification of formula \eqref{a5} is immediate, one simply takes the derivative in \eqref{a4}.

\bigskip

\subsubsection{Horizontal lift, $\mu$ and $\eta$}

One can define the lift $V$ of $\frac{\d}{\d t}$ with respect to the Poincaré type metric $\om_E$, cf. \cite{Siu86, Bo11, Schum}. It is a vector field of type $(1,0)$ on $\cX^\circ$ such that $dp$ 
maps it to $\frac{\d}{\d t}$ pointwise on $\cX^\circ$ and which is orthogonal to $T^{1,0}X_t$ for any $t$. 
In local coordinates, one has the following formula 
\begin{equation}\label{vecform}
V =\frac{\d}{\d t} -\sum_{\alpha, \beta} g^{\bar \beta \alpha }g_{t \bar \beta}\frac{\d}{\d z_{\alpha}} .
\end{equation}
Let $u$ be a holomorphic section of $p_{*}(\O (K_{\cX/D}+ L) \otimes \mathcal{I} (h_L))$. 
One can choose an arbitrary representative $U_0$ of $u$, this is an $L$-valued $(n,0)$-form on $\cX$ which coincides with $u_t$ on $X_t$. Now, let 
\begin{equation}
\label{uu}
\uu:=V\lrcorner \, (dt \wedge U_0).
\end{equation}
One can write locally (using the previous system of coordinates): $$U_0\wedge dt=a(z,t) dt \wedge dz_1\wedge \ldots \wedge dz_n $$
where $a(z,t)$ is holomorphic with values in $L$. We have an explicit formula: 
\begin{equation}\label{explicit}
 \uu=a(z,t) \big(dz_1\wedge \ldots \wedge dz_n-\sum_{\alpha,\beta} (-1)^{\alpha}g^{\bar \beta \alpha }g_{t \bar \beta} dt\wedge dz_1\wedge \ldots\wedge \widehat{dz_{\alpha}} 
 \wedge \ldots \wedge dz_n\big).
\end{equation}

\medskip

\noindent By construction, we have 
\begin{equation}\label{add1}
dt \wedge \uu= dt \wedge U_0. 
\end{equation}
Therefore, although $\uu$ is only well defined on $\cX^\circ$, $\uu\wedge d t$ can be extended 
as a smooth form on $\cX$. 

\begin{remark}\label{prim} {\rm The representative $\uu$ in \eqref{explicit} has the following interesting property
\begin{equation}\label{a10}
\uu\wedge \om_E|_\Omega= a(z, t)g_{t\bar t} \, dt\wedge d\ol t\wedge dz_1\wedge \ldots \wedge dz_n.
\end{equation}
In particular, we have $\displaystyle \frac{\uu\wedge \om_E}{dt}\Big|_{X_t}= 0$ for every $t$.
}
\end{remark}

\smallskip

\noindent Moreover, as $U_0$ is a smooth representative of $u$ on $\cX$ and $u$ is a holomorphic form, we have $\bar\partial (U_0 \wedge dt)=0$ on $\cX$.
Combining with \eqref{add1}, we know that $\bar\partial \uu\wedge dt =0$ on $\cX^\circ$.
As a consequence, we can find a smooth $(n-1, 1)$-form $\eta $ on $\cX^\circ$ such that
\begin{equation}\label{add2}
 \bar\partial \uu=dt \wedge \eta \qquad\text{on } \cX^\circ
\end{equation}
and one has 
\begin{align}
\label{add3}
D' \uu&\underset{\rm loc}{=}\d \uu  -\d \varphi_L \wedge \uu\\
\nonumber &= dt \wedge \mu
\end{align}
for some $(n,0)$-form $\mu$ on $\cX^\circ$.

\subsubsection{Estimates for $\eta$ and $\mu$.}
\label{amis}

\noindent In this section our goal is to establish the following statement.

\begin{lemme}\label{meta}
We consider the form $\displaystyle \eta|_{X_t}, \mu|_{X_t}$ 
induced by the representative $\uu$ constructed in the previous section in \eqref{add2} and \eqref{add3} respectively. Then $\displaystyle \eta|_{X_t}, \mu|_{X_t}$
as well as $\displaystyle \dbar \mu|_{X_t}$ are $L^2$ 
with respect to $(\om_E, e^{-\phi_L})$.
\noindent

Moreover, $\uu, \eta, \mu$ and $\dbar \mu$ are also in $L^{2}(\cX^\circ)$ with respect to $(\om_E, e^{-\phi_L})$, say up to shrinking $D$.
\end{lemme}

\begin{proof} This is routine: we can easily obtain the explicit expression of $\eta$ and $\mu$, and then we simply evaluate their respective $L^2$ norms by using the estimates \eqref{a4}--\eqref{a5}.
We detail to some extent the calculations next.
\smallskip

\noindent First of all, we have
\begin{eqnarray*}
\eta|_{X_t} = a (z,t) \sum_{\alpha,\beta, r} (-1)^{\alpha+ 1}(g^{\bar \beta \alpha }_{,\bar r}g_{t \bar \beta}+ g^{\bar \beta \alpha }g_{t \bar \beta,\bar r}) \, 
dz_1\wedge \ldots\wedge \widehat{dz_{\alpha}}  \wedge \ldots \wedge dz_n\wedge d\bar z_{r} .
\end{eqnarray*} 
where we use the notation $\displaystyle g^{\bar \beta \alpha }_{,\bar r}:= \frac{\partial g^{\bar \beta \alpha }}{\partial \ol z_r}$.

\noindent We consider first the quantity 
\begin{equation}\label{a11}
\left|g^{\bar \beta \alpha }_{,\bar r}g_{t \bar \beta}\right|^2.
\end{equation}
Thanks to the equality \eqref{a4}, up to a constant it is smaller than 
\begin{equation}\label{a12}
\frac{1}{|f_\beta|^2}\left( \delta_{\alpha r}\delta_\alpha \log^2 |z_\alpha|^2|f_\beta|^2+ \delta_{\beta r}\delta_\beta|f_\alpha|^2+\Big(1-\delta_r+\frac{\delta_r}{|z_r|^2 \log^2 |z_r|^2}\Big)|f_\alpha f_\beta|^2\right)
\end{equation}
which simplifies to 
$$|f_\alpha|^2+\delta_r\Big[\delta_{\alpha r} \log^2 |z_r|^2+\delta_{\beta r} \frac{|f_\alpha|^2}{|f_r|^2}+\Big(\frac{1}{|z_r|^2 \log^2 |z_r|^2}-1\Big)|f_\alpha|^2\Big].$$ 
Since 
\[\left|dz_1\wedge \ldots\wedge \widehat{dz_{\alpha}}  \wedge \ldots \wedge dz_n\wedge d\bar z_{r}\right|^2_{\om_E}dV_{\om_E} \lesssim \frac{|f_r|^2}{|f_\alpha|^2}dV_\om\]
we eventually find
\begin{equation} \label{b1}
\left|g^{\bar \beta \alpha }_{,\bar r}g_{t \bar \beta}\right|^2\left|dz_1\wedge \ldots\wedge \widehat{dz_{\alpha}}  \wedge \ldots \wedge dz_n\wedge d\bar z_{r}\right|^2_{\om_E}dV_{\om_E}\leq C(1+\delta_r\delta_{\alpha r}\log^2|z_r|^2)\, dV_\om.
\end{equation}

\noindent The term 
\begin{equation} \label{b2}
\left|g^{\bar \beta \alpha }g_{t \bar \beta,\bar r}\right|^2
\end{equation}
is bounded by 
\begin{equation} \label{b3}
\delta_{\beta r}\delta_\beta(1+ \log^2|z_\beta|^2)\frac{|f_\alpha|^2}{|f_\beta|^2}+ \left(1-\delta_r+ \frac{\delta_r}{|w_r|^2}\right)|f_\alpha|^2
\end{equation}
and we see that the same thing as before occurs, i.e. 
\begin{equation} \label{b4}
\left|g^{\bar \beta \alpha }g_{t \bar \beta,\bar r}\right|^2\left|dz_1\wedge \ldots\wedge \widehat{dz_{\alpha}}  \wedge \ldots \wedge dz_n\wedge d\bar z_{r}\right|^2_{\om_E}dV_{\om_E}\leq C(1+\delta_r\log^2|z_r|^2)\, dV_\om.
\end{equation}
Thus, the restriction of $\eta$ to any fiber of the family $p:\cX\to D$ is $L^2$ with respect to $(\om_E, e^{-\phi_L})$ since the holomorphic function $a(z, t)$ belongs to the multiplier ideal sheaf defined by $h_L$. Indeed, setting $\nu:=|a|^2dV_\omega$, one has $e^{-\varphi_L}\in L^{1+\ep}(\nu)$ for some $\ep>0$ (this is easily checked since $\varphi_L$ has analytic singularities and $e^{-\varphi_L}\in L^1(\nu)$) while $\log |z_r|^2\in L^p(\nu)$ for all $p>0$, so that Hölder inequality shows the claim.
\medskip

\noindent 
The local expression of the form $\displaystyle \mu |_{X_t}$ is obtained by restricting $\displaystyle \frac{D'  \uu}{dt}$ to the fiber $X_t$; it reads as 
\begin{align} 
\frac{\mu |_{X_t}}{dz}= & \label{b5}
\, a_{,t} -a \sum_{\alpha,\beta} (g^{\bar \beta \alpha }_{,\alpha}g_{t \bar \beta}+g^{\bar \beta \alpha }g_{t \bar \beta, 
\alpha}) - \sum_{\alpha,\beta} a_{, \alpha} g^{\bar \beta \alpha }g_{t \bar \beta} \\
 &- a(z, t)\varphi_{L, t}+ a(z,t)\sum_{\alpha, \beta}\varphi_{L, \alpha} g^{\bar \beta \alpha }g_{t \bar \beta} \nonumber \\
\nonumber
\end{align}
where $dz:= dz_1\wedge \dots \wedge dz_n$.

\noindent By our transversality conditions, the function $a_{,t}$ is still $L^2$ with respect to 
$h_L$. The term $\displaystyle g^{\bar \beta \alpha }_{,\alpha}g_{t \bar \beta}+g^{\bar \beta \alpha }g_{t \bar \beta, \alpha}$ is treated as we did for \eqref{a11} and \eqref{b2}, with the exception that the indexes $r$ and $\beta$ coincide (and the type of the form is different). We have up to some constant
{\small
\begin{align} 
\label{b6} \left|g^{\bar \beta \alpha }_{,\alpha}g_{t \bar \beta}\right| &\leq \,
\left((1-\log |z_\alpha|^2) |f_\beta| + \delta_{\alpha_\beta}\delta_\alpha |f_\alpha|+\Big[(1-\delta_\alpha)+\frac{\delta_\alpha}{|z_\alpha| \log^2|z_\alpha|^2}\Big] \cdot |f_\alpha f_\beta|\right) \cdot \frac{1}{|f_\beta|}\\
\nonumber & \lesssim 1-\log |z_\alpha|^2
\end{align}}
and we can bound this term as before. The second term satisfies $g^{\bar \beta \alpha }g_{t \bar \beta, 
\alpha}=g^{\bar \beta \alpha }g_{\alpha \bar \beta, t}$ since $\om_E$ is Kähler, hence it is bounded. 

\noindent Next, we have $\displaystyle \left|g^{\bar \beta \alpha }g_{t \bar \beta}\right|^2\leq C|z_\alpha|^2\log^2|z_\alpha|$, and also that 
\begin{equation} \label{b77}
|a_{, \alpha}|^2|z_\alpha|^2\log^2|z_\alpha|^2 e^{-\varphi_L}=\mathcal O (|a|^2 \log^2 |z_\alpha|^2 e^{-\varphi_L}) \in L^1
\end{equation}
for any $\alpha= 1,\dots, p$ again by the transversality/$L^2$ conditions we impose to $a(z, t)$, so the third term in \eqref{b5} is in $L^2$. A similar argument applies to the second line of \eqref{b5}.
\medskip

\noindent The last part of the proof of our lemma concerns $\dbar \mu$; the computations are using \eqref{a5}. We will only discuss the term
\begin{equation} \label{b8}
\frac{\partial}{\partial \ol z_r} \left(g^{\bar \beta \alpha }_{,\alpha}g_{t \bar \beta}+g^{\bar \beta \alpha }g_{t \bar \beta, 
\alpha}\right)dz\wedge d\ol z_r
\end{equation}
since for \eqref{b8} the computations are the most involved. The reason why we are able to conserve the $L^2$ property is that the partial derivative with respect to $\ol z_r$ will induce a new term of order
$\mathcal O(1/|z_r|)$ if $r\leq p$, and its square will be compensated by $|d\ol z_r|^2_{\om_E}$.
As for the computations:
the singularities induced by $\displaystyle g^{\bar \beta \alpha }_{,\alpha \bar r}g_{t \bar \beta}$ are bounded by the following quantity
\begin{align}
\left|g^{\bar \beta \alpha }_{,\alpha \bar r}g_{t \bar \beta}\right|\leq & \, \delta_\alpha\delta_{r\alpha}\frac{1}{|z_r|}+ \frac{\delta_\alpha\delta_{\beta r}}{|f_\beta|}\log\frac{1}{|z_\alpha|^2}\log\frac{1}{|z_\beta|^2}+ \delta_\alpha\delta_{r\alpha\beta}\frac{1}{|f_\beta||z_\alpha|^2}\\
 \nonumber &+ \frac{\delta_r}{|w_r|}\left(\delta_{\alpha}\log\frac{1}{|z_\alpha|^2}+ \delta_{\alpha}\delta_{\alpha\beta}\right)\\
\nonumber & +\frac{\delta_{\alpha}\delta_{r}}{|w_\alpha w_r|}|f_\alpha| + \frac{\delta_\alpha}{|w_\alpha|}\left(\delta_\beta\delta_{\beta r}\log\frac{1}{|z_\beta|^2}\frac{|f_\alpha|}{|f_\beta|}+ \delta_\alpha\delta_{\alpha r}\right)
\end{align}
from which we see that the first part of \eqref{b8} is $L^2$. The remaining terms are 
\begin{equation}
\label{three}
\left(g^{\bar \beta \alpha }_{,\alpha}g_{t \bar \beta, \bar r}+ g^{\bar \beta \alpha }_{,\bar r}g_{t \bar \beta, \alpha}+ g^{\bar \beta \alpha }g_{t \bar \beta, 
\alpha\bar r}\right)dz\wedge d\ol z_r
\end{equation}
for which one could use the fact that the metric $\om_E$ is K\"ahler
and so we have
\begin{equation} \label{b7}
g_{t \bar \beta, \alpha}= g_{\alpha \bar \beta, t}, \qquad g_{t \bar \beta, \alpha\bar r}= g_{\alpha \bar \beta, t\bar r}.
\end{equation}
The equalities \eqref{b7} are simplifying a bit the calculations, since the derivative with respect to $t$ does not increase at all the order of the singularity.

For the first term of \eqref{three}, we have, up to a multiplicative constant

\begin{align*}
\left|g^{\bar \beta \alpha }_{,\alpha} g_{t \bar \beta, \bar r}\right| \le &
\left| \delta_\alpha (1-\log|z_\alpha|^2) \bar f_\beta +\delta_\alpha \delta_{\alpha\beta} f_\alpha+ \left( (1-\delta_\alpha)+\frac{\delta_\alpha}{z_\alpha \log^2|z_\alpha|^2} \right) f_\alpha \bar f_\beta\right|\times \\
& \left| \delta_r\delta_{r\beta}\frac{1}{|z_r|^2\log^2 |z_r|^2}+\frac{1}{\bar f_\beta}\left(1+\frac{\delta_r}{\bar z_r \log^2 |z_r|^2}\right)\right|\\
\le & \left| \left(1-\delta_\alpha \log |z_\alpha|\right) \cdot \left( 1+\frac{\delta_r}{|z_r|(-\log |z_r|^2)}) \right) \right|
 \end{align*}
 In particular, we get 
 \[\left|g^{\bar \beta \alpha }_{,\alpha} g_{t \bar \beta, \bar r}\right| \cdot |d\bar z_r|_{\om_E} \lesssim 1-\delta_\alpha \log |z_\alpha|\]
 and we are done with this term as before.
 
 For the second term of \eqref{three}, we have, using \eqref{b7}
 {\small
 \begin{align*}
 \left| g^{\bar \beta \alpha }_{,\bar r}g_{\alpha \bar \beta, t}\right| \le &
 \frac{1}{|f_\alpha f_\beta|}\cdot \Bigg(\delta_{\alpha r}\delta_\alpha(1-\log |z_\alpha|^2)| f_\beta|+\delta_{\beta r}\delta_\beta |f_\alpha|\\
 &+ \Big[(1-\delta_r)+\delta_r (-\log |z_r|^2+\frac{1}{|z_r|\log^2|z_r|^2}\Big]|f_\alpha f_\beta|\Bigg) \\
 \le &1+\delta_{r}\frac{-\log |z_r|^2}{|z_r| (-\log |z_r|^2)}.
 \end{align*}
}
In particular, we get 
\[\left|g^{\bar \beta \alpha }_{,\bar r}g_{\alpha \bar \beta, t}\right| \cdot |d\bar z_r|_{\om_E} \lesssim 1-\delta_r \log |z_r|\]
and we are done. 

As for the last term of \eqref{three}, we use \eqref{b7} and \eqref{a5} to see that the expansion of
$  g^{\bar \beta \alpha }g_{t \bar \beta, \alpha\bar r}$ will only involve terms like
\[\psi_{\alpha\bar \beta, \bar r}, \frac{\partial_{\bar r} f_\alpha}{f_\alpha}, \frac{\partial_{\bar r} \bar f_\beta}{\bar f_\beta}\]
which are respectively of order \[\frac{\delta_r}{z_r \log |z_r|^2},\frac{\delta_r\delta_{\alpha r}}{z_r \log |z_r|^2}, \frac{\delta_r\delta_{\beta r}}{\bar z_r}.\]
All in all, we find 
\[\left| g^{\bar \beta \alpha }g_{t \bar \beta, \alpha\bar r}\right| \cdot |d\bar z_r|_{\om_E} \lesssim 1-\delta_r \log |z_r|\]
and this is the end of the main part of the proof.
\medskip

\noindent The integrability of $\uu, \eta, \mu$ on $\cX^\circ$ follows directly from the estimates we have obtained above. Concerning $\dbar \mu$ there is one additional term given by 
$\displaystyle \frac{\partial}{\partial \ol t}$ of the expression in \eqref{b5}. This is however harmless:
given the shape of the coefficients $(g_{\alpha\ol\beta})$ (i.e. the transversality conditions), the additional anti-holomorphic derivative with respect to $t$ induces no further singularity and the estimates e.g. for the term
$$\frac{\partial}{\partial \ol t}\left(g^{\bar \beta \alpha }_{,\alpha}g_{t \bar \beta}\right)$$
will be completely identical to those already obtained 
$\displaystyle g^{\bar \beta \alpha }_{,\alpha}g_{t \bar \beta}$. We leave the details to the interested reader. 
\end{proof}

\begin{remark}
Using quasi-coordinates adapted to the Poincaré metric $\om_E$ (cf. e.g. \cite{CY3, KobR, TY87}), we can prove easily that $\eta$ and its derivatives are in $L^2$. However, that argument cannot be applied to $\mu$ because of the singularity in the Chern connection of $(L,h_L)$.
\end{remark}

\subsection{A few results from $L^2$ Hodge theory}\label{hodge}

We recall briefly a few results of $L^2$-Hodge theory for a complete manifold endowed with a Poincaré type metric, following closely \cite{JCMP}. We are in the following setting.

Let $X$ be a $n$-dimensional compact Kähler manifold, and let $(L, h_L)$ be a  
line bundle  endowed with a (singular) metric $h_L= e^{-\phi_L}$ such that
\begin{enumerate}[label=$\bullet$]
\item $h_L$ has analytic singularities;
\item Its Chern curvature satisfies $\displaystyle i\Theta_{h_L}(L) \geq 0$ in the sense of currents.
\end{enumerate} 

\bigskip 

 We consider a modification $\pi: \wh X\to X$ of $X$ such that the support of the singularities of 
$\varphi_L\circ \pi$ is a simple normal crossing divisor $E$. As usual, we can construct $\pi$ such that its restriction to $\wh X\setminus E$ is an biholomorphism. Then 
\begin{equation}\label{eq30}
\varphi_L\circ \pi|_{\Omega}\equiv \sum_{\alpha=1}^p e_\alpha\log |z_\alpha|^2
\end{equation}
modulo a smooth function. Here $\Omega\subset \wh X$ is a coordinate chart, and 
$(z_\alpha)_{\alpha=1,\dots, n}$ are coordinates such that
$E\cap \Omega= (z_1\dots z_p= 0)$.

\noindent Let $\wh \omega_E$ be a complete K\"ahler metric on $\wh X\setminus E$,
with Poincar\'e singularities along $E$, and let 
\begin{equation}\label{eq31}
\omega_E:= \pi_\star(\wh \omega_E)
\end{equation}
be the direct image metric. We note that in this way $(X^\circ, \omega_E)$ becomes a 
complete K\"ahler manifold, where $X^\circ := X\setminus (h_L= \infty)$. 
\smallskip

\begin{remark}
\label{extension}
If $u$ is a $L$-valued $(p,0)$-form on $X^\circ$ which is $L^2$ with respect to $\om_E$, then it is also $L^2$ with respect to an euclidean metric on $\hat X$ (or $ X$, too).Therefore, if $u$ is holomorphic, then it extends holomorphically to  $X$ and more generally any smooth compactification of $X^\circ$.  
\end{remark}

\noindent The main goal of this section is to establish the following decomposition theorem, which is a slight generalization of the corresponding result in \cite{JCMP}.

\begin{thm}\label{Hodge}
	Consider a line bundle $(L, h_L)\to X$ endowed
	with a metric $h_L$ with analytic singularities, as well as
	the corresponding complete K\"ahler manifold $(X^\circ, \omega_E)$, 
	cf. \eqref{eq31}.   If $i\Theta_{h_L} (L) \geq 0$ on $X$, we have the following Hodge decomposition
		$$
		L^2_{n, 1}(X^\circ, L)= {\mathcal H}_{n,1}(X^\circ, L)\oplus
		{\Imm \dbar}\oplus \Imm \dbar^\star.
		$$
Here ${\mathcal H}_{n,1}(X^\circ, L)$ is the space of $L^2$ $\Delta^{''}$-harmonic $(n,1)$-forms.	
\end{thm}

\noindent The proof follows closely the aforementioned reference, in which the case $\Theta_{h_L}(L)=0$ is treated.
For the sake of completeness, we will sketch the proof and highlight the differences. 
\medskip

\noindent We start by recalling the following result.
\begin{lemme}\label{cutoff}
	There exists a family of smooth functions $(\mu_\ep)_{\ep> 0}$ with the following properties.
	\begin{enumerate}
		
		\item[\rm (a)] For each $\ep> 0$, the function $\mu_\ep$ has compact support in $X^\circ$, and 
		$0\leq \mu_\ep\leq 1$.
		
		\item[\rm (b)] The sets $(\mu_\ep= 1)$ are providing an exhaustion of $X^\circ$.
		
		\item[\rm (c)] There exists a positive constant $C> 0$ independent of $\ep$ such that 
		we have $$\displaystyle \sup_{X^\circ}\left(|\partial \mu_\ep|_{\omega_E}^2+ |\ddbar \mu_\ep|^2_{\omega_E}\right)\leq C.$$
	\end{enumerate} 
\end{lemme}

We have also the Poincaré type inequality for the $\dbar$-operator acting on $(p,0)$-forms.

\begin{proposition}\label{poinc}\cite{JCMP} 
	Let $\displaystyle (\Omega_j)_{j=1,\dots, N}$ be a finite union of coordinate sets of $\widehat X$ covering $E$, and let $\widehat U$ be any open subset contained in their union and $U:=\pi(\widehat U)$. 	Let $\tau$ be a $(p, 0)$-form with compact support 
	in a set $U\setminus \pi(E)\subset X$ and values in $(L, h_L)$. Then we have 
	\begin{equation}\label{eq57}
	\frac{1}{C}\int_U|\tau|_{\om_E}^2e^{-\phi_L}dV_{\omega_E}\leq \int_U|\dbar\tau|^2_{\omega_E}e^{-\phi_L}dV_{\omega_E}
	\end{equation}
	where $C$ is a positive numerical constant.
\end{proposition}

\noindent We emphasize that the constant $C$ in \eqref{eq57} only depends on the distortion between the model Poincar\'e metric on $\Omega_j$ with singularities on $E$ and 
the global metric $\widehat \omega_E$ restricted to $\Omega_j$.
Another important observation is that by using the cut-off function $\mu_\ep$ in Lemma \ref{cutoff}, we infer that \eqref{eq57} holds in fact for any $L^2$-bounded form with compact support in $U$. 

\medskip


$\bullet$ {\it Quick recap around the Bochner-Kodaira-Nakano formula.} 

\noindent
We recall the following formula, which is central in complex differential geometry
\begin{equation}
\label{bochner}
\Delta''=\Delta'+[i\Theta_{h_L} (L), \Lambda_{\omega_E}]
\end{equation}
where $\Delta''= \dbar \dbar^*+\dbar^*\dbar$ and $\Delta'=D'D'^*+D'^*D'$ where $D'$ is the $(1,0)$-part of the Chern connection on $(L,h_L)$. Let us also recall the well-known fact that the self-adjoint operator 
\[A:=[i\Theta_{h_L} (L), \Lambda_{\omega_E}]\] is semi-positive when acting on $(n,q)$ forms, for any $0\le q \le n$ as long as $i\Theta_{h_L}(L) \ge 0$. 
An immediate consequence of \eqref{bochner} is that for a $L^2$-integrable form $u$ with values in $L$ of any type in the domains of $\Delta'$ and $\Delta''$, we have
\begin{equation}
\label{bochner2}
\|\dbar u\|^2_{L^2}+\|\dbar^*u\|^2_{L^2}=\|D' u\|^2_{L^2}+\|D'^*u\|^2_{L^2}+\int_{X^\circ}\langle Au,u\rangle dV_{\omega_E}
\end{equation}
where $\|\cdot\|_{L^2}$ (resp. $\langle, \rangle$) denotes the $L^2$-norm (resp. pointwise hermitian product) taken with respect to $(h_L, \om_E)$. 
Let $\star:\Lambda^{p,q}T_{X^\circ}^*\to \Lambda^{n-q,n-p}T_{X^\circ}^*$ be the Hodge star with respect to $\om_E$; we introduce for any integer $0\le p \le n$ the space
\begin{equation}
\label{Hp}
H^{(p)}:=\{F\in  H^0(X^\circ, \Omega_{X^\circ}^{p}\otimes L)\cap L^2;  \,\, \int_{X^\circ}\langle A\star F, \star F\rangle dV_{\om_E}=0\}
\end{equation}
and we can observe by Bochner formula that for a $L^2$ integrable, $L$-valued $(p,0)$-form $F$, one has
\begin{equation}
\label{hp}
 \Delta''(\star F)=0 \Longleftrightarrow \Delta'(\star F)=0 \text{ and } \int_{X^\circ}\langle A\star F, \star F\rangle dV_{\om_E}=0 \Longleftrightarrow F\in H^{(p)}  .
\end{equation}
\medskip

\noindent The proof of Theorem~\ref{Hodge}, which we give below, makes use of the
following proposition which is the $\dbar$-version of the Poincar\'e inequality established in \cite{Auv17}. 

\begin{proposition}\label{cor:2}
	Let $p\leq n$ be an integer.  There exists a positive constant $C> 0$ such
	that the following inequality holds
	\begin{equation}\label{eq3233}
	\int_{X^\circ}|u|^2_{\omega_E}e^{-\phi}dV_{\omega_E}\leq C \left(\int_{X^\circ}|\dbar u|^2_{\omega_E}e^{-\phi}dV_{\omega_E} +
	\int_{X^\circ} \langle A \star u , \star u \rangle d V_{\omega_E}\right)
	\end{equation}
	for any $L$-valued form $u$ of type $(p, 0)$ which belongs to the domain of $\dbar$ and which is orthogonal to the  space $H^{(p)}$ defined by \eqref{Hp}. 	Here $\star$ is the Hodge star operator with respect to the metric $\omega_E$.
\end{proposition}  

\begin{proof}[Proof of Proposition~\ref{cor:2}]
If a positive constant as in \eqref{eq3233} does not exists, then we obtain a sequence $u_j$ of 
$L$-valued forms of type $(p,0)$ orthogonal to $H^{(p)}$ such that 
\begin{equation}
\int_{X^\circ}|u_j|^2_{\omega_E}e^{-\phi}dV_{\omega_E}= 1, \qquad
\lim_j\int_{X^\circ}|\dbar u_j|^2_{\omega_E}e^{-\phi_L}dV_{\omega_E}= 0
\end{equation}
and 
$$\lim_j \int_{X^\circ} \langle A \star u_j , \star u_j \rangle   d V_{\omega_E} =0.$$
It follows that the weak limit $u_\infty$ of $(u_j)$ is holomorphic and belongs to $H^{(p)}$. 
On the other hand, each $u_j$ is perpendicular to $H^{(p)}$, so it follows
that $u_\infty$ is equal to zero.

Let us first show that the weak convergence $u_i\rightharpoonup u_\infty$ also takes places in $L^2_{\rm loc}(X^\circ)$. To that purpose, let us pick a small Stein open subset $U \Subset X^\circ$. By solving the $\dbar$-equation $U$, we can find $w_j$ such that $\dbar w_j = \dbar u_j$ on $U$ and $\int_U |w_j|^2 \rightarrow 0$. 
 Therefore $u_j -w_j$ is holomorphic on $U$ and converges weakly, hence strongly to $u_\infty |_U$. In particular $u_j$ converges to $u_\infty$ in $L^2$ on $U$. As $u_\infty =0$, we have 
	\begin{equation}\label{eq611}
	u_j |_K\to 0
	\end{equation}
	in $L^2$ for any compact subset $K\subset X^\circ$.
	
	The last step in the proof is to notice that the considerations above contradict the fact that the $L^2$ norm of each $u_j$ is equal to one. This is not quite immediate, but is precisely as the end of the proof of Lemma 1.10 in
	\cite{Auv17}, so we will not reproduce it here. The idea is however very clear: 
	in the notation of Proposition~\ref{poinc}, we choose $V$ small enough so that it admits a cut-off function $\chi$ with small gradient with respect to $\om_E$. Then, we decompose each $u_j$ as $u_j= \chi u_j+ (1-\chi)u_j$. Then the $L^2$ norm of $\chi u_j$ is small
	by \eqref{eq611}. The $L^2$ norm of $(1-\chi)u_j$ is equally small 
	by \eqref{eq57}, and this is how we reach a contradiction.
\end{proof}

\noindent We have the following direct consequences of Proposition \ref{cor:2}.

\begin{cor}
	There exists a positive constant $C> 0$ such
	that the following inequality holds
	\begin{equation}\label{eq32334}
	\int_{X^\circ}|u|^2_{\omega_E}e^{-\phi}dV_{\omega_E}\leq C \left(\int_{X^\circ}|\dbar u|^2_{\omega_E}e^{-\phi}dV_{\omega_E}\right)
	\end{equation}
	for any $L$-valued form $u$ of type $(n, 0)$ which belongs to the domain of $\dbar$ and which is orthogonal to the kernel of $\dbar$. 
\end{cor}

\begin{proof} This follows immediately from Proposition \ref{cor:2} combined with the observation that the
  curvature operator $A$ is equal to zero in bi-degree $(n,0)$.
\end{proof}

\noindent The next statement shows that in bi-degree $(n, 2)$ the image of the operator $\dbar^\star$ is closed.

\begin{cor} There exists a positive constant $C> 0$ such that the following holds true.
	Let $v$ be a $L$-valued form of type $(n,2)$. We assume that $v$ is $L^2$, in the domain of $\dbar$ and orthogonal 
	to the kernel of the operator $\dbar^\star$. Then we have
	\begin{equation}\label{eq62}
	\int_{X^\circ}|v|^2_{\omega_E}e^{-\phi_L}dV_{\omega_E}\leq C \int_{X^\circ}|\dbar^\star v|^2_{\omega_E}e^{-\phi_L}dV_{\omega_E}.
	\end{equation}
\end{cor}
\smallskip

\begin{proof}
Let us first observe that the Hodge star $u:= \star v$, of type $(n-2,0)$, is orthogonal to $H^{(n-2)}$. This can be seen as follows. Let us pick $F\in H^{(n-2)}$; it follows from \eqref{hp} that we have $\displaystyle \dbar^\star(\star F)= 0$. In other words, $\star F \in \ker \dbar^\star$.
We thus have
$$\int_{X^\circ} \langle u, F\rangle   d V_{\omega_E} = \int_{X^\circ} \langle v, \star F \rangle  d V_{\omega_E} =0.$$
Applying Bochner formula \eqref{bochner2} to $v$ and using the facts that $\dbar  v=0$ (since $v$ is orthogonal to $\ker \dbar^*$) and that $\dbar^*u=0$ for degree reasons, we get
\begin{equation}\label{eq63}
\|\dbar^*v\|^2_{L^2}=\|\dbar u\|^2_{L^2}+\int_{X^\circ}\langle A \star u, \star u\rangle dV_{\omega_E}
\end{equation}
This proves the corollary by applying Proposition \ref{cor:2}.
\end{proof}
\medskip

\noindent We discuss next the relative version of the previous estimates. Let $p: \cX\to D$ and $(L, h_L)$ be the family of manifolds and the line bundle, respectively fixed in the previous section. We assume that 
\begin{equation}\label{e3}
D\ni t\mapsto \dim\big(\ker(\Delta_t'')\big) \quad \mbox{is constant}
\end{equation}
where the Laplace operator $\Delta_t''$ is the one acting on $L^2$ $(n,1)$-forms with respect to $(\omega_E, h_L)$.
\smallskip

\noindent The next result is a consequence of the proof of Proposition \ref{cor:2}.

\begin{cor}\label{unifconst}
	Under the additional assumptions \eqref{e3} and {\rm \ref{A1}}, there exists a constant $C> 0$ independent of $t$ such that 
	\begin{equation}\label{uniconsts}
\int_{X_t ^\circ}|u|^2_{\omega_E}e^{-\phi}dV_{\omega_E}\leq C \left(\int_{X_t ^\circ}|\dbar u|^2_{\omega_E}e^{-\phi}dV_{\omega_E} +
\int_{X_t ^\circ} \langle A \star u , \star u \rangle d V_{\omega_E}\right)
	\end{equation}
	for all $L^2$ forms $u$ orthogonal to the space $H^{(p)} _t$ defined in \eqref{Hp} on the fiber $X_t$.
\end{cor}

Here the constant $C$ is uniform in the sense that for any subset $U$ of compact support in $D$, we can find a constant $C$ depending on $U$ such that \eqref{uniconsts} is satisfied for any $t\in U$

\begin{proof} We first show that every form $F_0$ on the central fiber which is in the space $H^{(p)} _0$ can be written as limit of $\displaystyle F_{t_i} \in H^{(p)} _{t_i}$. This is of course well-known in the compact case, but we include a proof here since we could not find a reference fitting in our context.
	
	\noindent Let $\displaystyle (F_t)_{t\in D^\star}$ any family of $L$-valued holomorphic $p$-forms on the fibers above the pointed disk $D^\star$ such that 
	\begin{equation}\label{e1}
	\int_{X^\circ _t}|F_t|^2_{\om_E}e^{-\phi_L}dV_{\om_E}= 1.
	\end{equation}
	Then we can definitely extract a limit $F_\infty$ on the central fiber $X_0$, but in principle it could happen that $F_\infty\equiv 0$ is identically zero. Such assumption would lead however to a contradiction, as follows.
	
	\noindent We write locally on a coordinate chart $\Omega$ for $\cX$ 
	\begin{equation}\label{e2}
	F_t|_\Omega= \sum f_Idz_I\otimes e_L,
	\end{equation}
	where the coefficients $f_I$ are holomorphic, and of course depending on $t$. We can assume that the 
	multiplier ideal sheaf of $h_L$ is trivial, given the transversality conditions that we have imposed
	(we can simply divide $F_t$ with the corresponding sections). If the weak limit of $F_t$ is zero, we can certainly extract a limit in strong sense, because the $L^2$ norm with respect to a smooth metric is smaller than the $L^2$ norm with respect to Poincar\'e metric, cf. also Remark~\ref{extension}.
	
	In this case, the $\sup$ norm of the coefficients $f_I$ above converges to zero as $t\to 0$. Since the Poincar\'e metric we are using has uniformly bounded volume, the equality \eqref{e1} will not be satisfied as soon as $t\ll 1$.
	\smallskip
	
	\noindent We now take an orthonormal basis $(F_{t, j})$ of the space $H_t^{(p)}$ (this is obtained by the 
	$\star_t$ of an orthonormal basis for the $\ker(\Delta''_t)$, for example). The previous considerations will allow us to construct by extraction an orthonormal family $(F_{\infty, j})$ in $H_0^{(p)}$; this will be a basis because of dimension considerations.
	\medskip
	
	\noindent We argue by contradiction and assume that the smallest constant $C_t$ for which \eqref{eq3233} holds true to for the fiber 
	$X_t$ tends to infinity when $t\to 0$. Then we get $u_i$ 
	on $\displaystyle X_{t_i}$ such that $u_i$ is orthogonal to the space $H_{t_i}^{(p)}$ and such that
	\begin{equation}\label{eq61}
	\int_{X_{t_i}}|u_i|^2_{\omega_E}e^{-\phi_L}dV_{\omega_E}= 1, \qquad
	\lim_i\int_{X_{t_i}}|\dbar u_i|^2_{\omega_E}e^{-\phi_L}dV_{\omega_E}= 0 .
	\end{equation}
	Then the limit 
	$$u_0 := \lim_{i\rightarrow 0} u_i$$ 
	is still orthogonal to $H_0^{(p)}$: this is exactly where the previous considerations are needed. 
	The rest of the proof of the corollary follows the arguments already given for 
	Proposition \ref{cor:2}, so we simply skip it.
\end{proof}

Now we can prove Theorem \ref{Hodge}.
\begin{proof}[Proof of Theorem~\ref{Hodge}]
	This statement is almost contained in \cite[chapter VIII, pages
	367-370]{bookJP}. Indeed, in the context of complete manifolds one has the
	following decomposition
	\begin{equation}\label{eq33}
	L^2_{n, 1}(X^\circ, L)= {\mathcal H}_{n,1}(X^\circ, L)\oplus \overline{\Imm \dbar}\oplus \overline{\Imm \dbar^\star}.    
	\end{equation}
	We also know (see \emph{loc. cit.}) that the adjoints $\dbar^\star$ and
	$D^{\prime\star}$ in the sense of von Neumann coincide with the formal
	adjoints of $\dbar$ and $D^\prime$ respectively.
	
	It remains to show that the range of the $\dbar$ and
        $\dbar^\star$-operators are closed with
	respect to the $L^2$ topology. In our set-up, this is a consequence of the
	particular shape of the metric $\omega_E$ at infinity (i.e. near the support
	of $\pi (E)$): we are simply using the inequalities \eqref{eq32334} and \eqref{eq62}. The former shows that the image of $\dbar$ is closed, and the latter does the same for $\dbar^\star$.
\end{proof}

\bigskip

We finish this section with the following result (relying of the decomposition theorem obtained above), identifying the $L^2$-integrable $\Delta''$-harmonic forms of bi-degree $(n,1)$ on $(X^\circ, \om_E, h_L)$ with the vector space $H^1(X,K_X\otimes L\otimes \mathcal I(h_L))$ \-- which is independent of $\om_E$. 

\begin{proposition}
\label{cohomology}
In the setting of Theorem~\ref{Hodge}, we have a natural isomorphism
\[\mathcal H_{n,1}(X^\circ, L)\overset{\simeq}{\longrightarrow} H^1(X,K_X\otimes L\otimes \mathcal I(h_L))\]
where $\mathcal H_{n,1}(X^\circ, L)$ is the space of $L^2$ integrable, $\Delta''$-harmonic $(n,1)$-forms on $X^\circ$. 
\end{proposition}

\begin{proof} We proceed in several steps.\\ 

$\bullet$ {\it Step 1. Reduction to the snc case.}

\noindent
The first observation is that the statement is invariant by blow-up whose centers lie on $X\setminus X^\circ$. It is obvious for the LHS while it follows from the usual formula $\pi_*(K_{X'}\otimes \pi^*L\otimes \mathcal I(\pi^*h_L))\simeq  K_X\otimes L\otimes \mathcal I(h_L)$ as well as Grauert-Riemenschneider vanishing  $R^1\pi_*(K_{X'}\otimes \pi^*L \otimes \mathcal I(\pi^*h_L))=0$ (see \cite[Cor.~1.5]{Matsumura16}) valid for any modification $\pi:X'\to X$. So from now on, we assume that the singular locus of $h_L$ is an snc divisor. In the following, we pick a finite Stein covering $(U_i)_{i\in I}$ of $X$.

\medskip

$\bullet$ {\it Step 2. Statement of the claim to solve the $\dbar$-equation.}

\noindent
Our main tool in the proof will be the following estimate

\begin{claim}
\label{claim dbar}
Let $v$ be a $(n,1)$-from on $X^\circ$ with values in $(L, h_L)$, and such that 
\begin{equation*}
\Delta'' v=0,\qquad \int_X|v|^2_{\om_E}e^{-\varphi_L}dV_{\om_E}< \infty.
\end{equation*}
Then for each coordinate set $\Omega\subset X$ there exists an $(n, 0)$-form $u$ on $\Omega$ such that 
\begin{equation}\label{l2ot}
\dbar u= v,\qquad \int_\Omega {|u|^2_{\om_E}}e^{-\varphi_L}dV_{\om_E}<\infty,
\end{equation}  
\end{claim}
For bi-degree reasons, the $(n,0)$-form $u$ in \eqref{l2ot} is $L^2$ with respect to $h_L$ (independently of any background metric). We postpone the proof of the claim for the moment and we will use it in order to prove Proposition~\ref{cohomology}. 

\medskip

$\bullet$ {\it Step 3. The map "harmonic to cohomology".}

\noindent We first construct an application
\begin{equation}\label{v0}
\Phi: \mathcal H_{n,1}(X^\circ, L) \longrightarrow \check{H}^1 (X, K_{X} \otimes L \otimes \mathcal I (h_L))
\end{equation}
as follows. Let $f\in  \mathcal H_{n,1}(X^\circ, L)$; by definition we have $\Delta'' f=0$. 
Therefore, on can solve on each $U_i^\circ:=U_i\cap X^\circ$ the equation $\dbar u_i=f$ where $u_i$ is an $L$-valued $(n,0)$-form on $U_i^\circ$ satisfies the condition \eqref{l2ot}. In particular, the form $u_{ij}:=u_i-u_j$ is a holomorphic $L$-valued $n$-form on $U_{ij}^\circ$ such that 
\[\int_{U_{ij}^\circ}{|u_{ij}|^2} e^{-\varphi_L} \le 2 \int_{U_{ij}^\circ} |f|^2_{\om_E}e^{-\varphi_L}dV_{\om_E} \]
It follows that $u_{ij}$ extends holomorphically across $E$ as a section of $K_{X}\otimes L \otimes \mathcal I(h_L)$ on $U_{ij}$ and therefore it defines a 
$1$-cocycle of the latter sheaf. It is straighforward to check that the class \[\Phi(f):=\{(u_{ij})_{i,j\in I}\}\in \check{H}^1(X,K_{X}\otimes L \otimes \mathcal I(h_L))\] is independent of the choice of the $L^2$-integrable form $u_i$ solving $\dbar u_i=f$.

\medskip
 
$\bullet$ {\it Step 4. The map "cohomology to harmonic".}

\noindent
Next, we have a natural morphism 
\begin{equation}
\label{v-1}
\Psi:  \check{H}^1 (X, K_{X} \otimes L \otimes \mathcal I (h_L))\longrightarrow  \mathcal H_{n,1}(X^\circ,L)
\end{equation}
Indeed, given a cocycle $v:=(v_{ij})_{i,j\in I}$ and a partition of unity $(\theta_k)_{k\in K}$  we use the Leray isomorphism and consider as usual the $L$-valued $(n,0)$-form
\begin{equation}\label{v1}
\tau_k:= \sum_{i\in I} \theta_i v_{ki} \quad \mbox{on } U_k
\end{equation}
and then the local $L$-valued forms of type $(n,1)$ 
$$\dbar \tau_k$$
are glueing on overlapping sets. Let $\beta_v$ be the resulting form. We have 
\begin{equation}\label{v2}
\dbar \beta_v=0 \qquad \mbox{and} \qquad \beta_v\in L^2,
\end{equation}
where the second property in \eqref{v2} is due to the fact that $\om_E\geq \om$. 
Under the canonical decomposition \[\ker \dbar = \mathcal H_{n,1} \overset{\perp}{\oplus} \Imm \dbar\] from Theorem~\ref{Hodge}, we define $\Psi(v)$ to be the orthogonal projection of $\beta_v$ onto $\mathcal H_{n,1}$.  It is clear that $\Psi$ above is well-defined: if $v_{ij}= v_i- v_j$, then $\tau_k-v_k$ is a global, $L^2$ form and our $\beta_v$ is exact and therefore its projection onto the kernel of $\Delta''$ is zero.

\medskip

$\bullet$ {\it Step 5. Compatibility of the maps.}

\noindent

We are left to showing that the maps $\Phi$ and $\Psi$ in \eqref{v0} and \eqref{v-1} are inverse to each other. Let $f\in \mathcal H_{n,1}$, $u_i\in L^2$ such that $\dbar u_i=f$ on $U_i^\circ$ and $u=(u_{ij})$. Then on $U_k^\circ$, one has \[\beta_u-f=\dbar(\sum_i \theta_i u_{ki}-u_k)=\dbar\left( -\sum_{i\in I} \theta_i u_i\right)\]
and that last form is globally exact in $X^\circ$ and $L^2$, hence $\Psi(\Phi(f))=f$. 

In the other direction, let $v:=(v_{ij})_{i,j\in I}$ be a cocycle and let us write $\beta_v=\Psi(v)+\dbar w$ for some $L^2$-integrable $(n,0)$-form $w$. On $U_k$, one has $\Psi(v)=\dbar(\tau_k-w)$ so that $\Phi(\Psi(v))$ is represented by the cocycle $(\tau_i-\tau_j)_{i,j\in I}=v$. 
\medskip

$\bullet$ {\it Step 6. Proof of Claim~\ref{claim dbar}.}

\noindent
In order to complete the proof of Proposition~\ref{cohomology}, we need to prove the Claim~\ref{claim dbar} that we used in the course of the proof. 

By \eqref{hp}, the form $\star v$ is holomorphic and its restriction to a coordinate subset $\Omega$ can be written as
\begin{equation}\label{v5}
\star v|_\Omega= \sum (-1)^{i-1}\alpha_i\wh dz_i\otimes e_L, \qquad \qquad \sum_j\int_\Omega \frac{|\alpha_j|^2}{|f_j|^2}e^{-\varphi_L}d\lambda< \infty
\end{equation}
where the $\alpha_i$ are holomorphic on $\Omega$ and $f_j$ is as in \eqref{a1}. 

\noindent Then we have
\begin{equation}\label{v6}
v|_\Omega= (-1)^n\sum_{i,k} \alpha_ig_{i\ol k} dz\wedge d\ol z_k\otimes e_L
\end{equation}
where $g_{i\ol k}$ are the coefficients of the metric $\om_E$.  The construction of the metric at the beginning shows that
\begin{equation}\label{v8}
g_{i\ol k}= \frac{\partial^2}{\partial z_i\partial \ol z_k}\left(\phi- \sum_j \log\log\frac{1}{|s_j|^2}\right)
\end{equation}
where $\phi$ is a local potential for the smooth metric $\omega$. Therefore we can get a primitive for $v|_\Omega$ by defining
\begin{equation}\label{v7}
u= \sum_{i} \alpha_i\frac{\partial}{\partial z_i}\left(\phi- \sum_j \log\log\frac{1}{|s_j|^2}\right)dz\otimes e_L.
\end{equation}  
By equality \eqref{v8} it verifies $\dbar u= v$ and in is also in $L^2$ as one can see by an direct explicit computation  combined with the second inequality in 
\eqref{v5}.

The proof of Proposition~\ref{cohomology} is now complete. 
\end{proof}

\section{Curvature formulas and applications}
\label{sec3}

In this section, we use the Set-up~\ref{altsetup}. We also borrow the Notation~\ref{notation} for the $L^2$ metric denoted by $h_\cF$ on the direct image bundle $\mathcal F=p_\star ( \O (K_{\cX/D}+L) \otimes \mathcal{I} (h_L))$ induced by $e^{-\phi_L}$.

Let $u \in H^0 (D, \mathcal F)$ and let $\uu$
be a $(n,0)$-form on $\cX^\circ$ representing $u$. 
Thanks to \eqref{add1}, for any smooth function $f (t)$ with compact support in $D$,
we have
\begin{equation}\label{intequl}
\int_D \|u\|^2_{h_\cF} \cdot \ddc f(t) = c_n\int_{\cX^\circ} \uu\wedge \bar\uu e^{- \phi_L} \wedge \ddc f(t) . 
\end{equation}
Recall that $h_\cF$ is smooth by Lemma \ref{prelemme}.

The aim of this section is to generalize formulas  \cite[(4.4), (4.8)]{Bo09} to our singular setting, cf Proposition~\ref{curvv} and Proposition~\ref{bobprop}.

\subsection{A general curvature formula}





\medskip

\noindent In this context we establish the following general formula, which generalise the corresponding result in \cite[(4.4)]{Bo09}.

\begin{proposition}\label{curvv}
Let $\uu$ be a continuous representative of $u$ such that:
\begin{enumerate}

\item[\rm (i)] $\uu, D' \uu$ 
and $\dbar (D'\uu)$ are
$L^2$ on $\cX^\circ$ with respect to $\omega_E, h_L$, 

\item[\rm (ii)] $\dbar \uu\wedge \overline{\dbar \uu}$ is $L^1$ on $\cX^\circ$ with respect to $\omega_E, h_L$.  
\end{enumerate}
Then the following formula holds true
\begin{align}
\label{comint}
\partial\bar\partial \|u\|^2 _{h_\cF}  \, = & \, \, c_n \Big[- p_\star ( (\Theta_{h_L}(L))_{\rm ac} \wedge \uu\wedge\bar \uu e^{-\phi_L}) + 
 (-1)^n p_\star (D' \uu\wedge \overline{D'\uu}e^{-\phi_L})\\
  &\, +  (-1)^n p_\star (\dbar \uu\wedge \overline{ \dbar\uu} \, e^{-\phi_L}) \Big] \nonumber
 \end{align}
 Here $(\Theta_{h_L}(L))_{\rm ac} $ is the absolutely continuous part of the current $\Theta_{h_L}(L)$.
 \end{proposition}

\begin{remark}
Here we merely require the $L^1$ integrability of $\dbar \uu\wedge \overline{\dbar \uu}$ and not the $L^2$-integrability of $\dbar \uu$. The reason is that for later application in  Theorem~\ref{mainresI}, we could only obtain the former condition. It is not clear whether the term $\dbar\uu$ in Theorem~\ref{mainresI} is $L^2$.
\end{remark}	

The proof of Proposition~\ref{curvv} will require a few preliminary computations and will be given on page~\pageref{page curvv} below. First, we start with the following result legitimizing integration by parts.

\begin{lemme}
\label{IPP}
If $\uu$ and $D' \uu$ are $L^2$ on $\cX^\circ$, we have
$$
\int_{\cX^\circ} \uu\wedge \bar\uu e^{- \phi_L } i\partial\bar\partial f(t)  = - \int_{\cX^\circ} 
D' \uu \wedge \bar\uu e^{-\phi_L} i\bar\partial f(t)  .$$ 
\end{lemme}

\begin{proof}
Let $\psi_\ep$ be the cut-off fonction in Lemma \ref{cutoff}.
Since $\uu$ is $L^2$ bounded with respect to $\phi_L$ and $\omega_E$, we have 
$$ \int_{\cX^\circ} \uu\wedge \bar\uu e^{- \phi_L} i\partial\bar\partial f(t)   = \lim_{\ep \rightarrow 0}\int_{\cX^\circ} 
\psi_\ep \uu\wedge \bar\uu e^{- \phi_L} i\partial\bar\partial f(t) . 
$$
An integration by parts yields 
{\small
\begin{align}
\label{firstee}
\int_{\cX^\circ} 
\psi_\ep \uu\wedge \bar\uu e^{- \phi_L} i\partial\bar\partial f(t) =&
 - \int_{\cX^\circ} i\partial \psi_\ep \wedge \uu\wedge \bar\uu e^{- \phi_L} \bar\partial f(t)  \\
\nonumber & -\int_{\cX^\circ} \psi_\ep \wedge D' \uu\wedge \bar\uu e^{-\phi_L} i\bar\partial f(t) - 
(-1)^n \int_{\cX^\circ} 
\psi_\ep  \wedge \uu\wedge \overline{ \dbar \uu}e^{-\phi_L} i\bar\partial f(t) .
\end{align}}
Since $\uu$ is a representative of a holomorphic section $u$, we know by \eqref{add2} that 
$ \dbar \uu= dt \wedge \eta$, hence
\begin{equation}\label{degreecon}
\dbar \uu\wedge dt =0
\end{equation}
and the third term of RHS of \eqref{firstee} vanishes.

The first term of RHS of \eqref{firstee} tends to $0$ because $\uu$ is assumed to be globally $L^2$-integrable. Similarly, we see that the second term of RHS of  \eqref{firstee}  tends to
$$-\int_{\cX^\circ} 
D' \uu \wedge \bar\uu e^{-\phi_L} i\bar\partial f(t) .$$
The lemma is thus proved.
\end{proof}

%

\noindent As a corollary of Lemma~\ref{IPP} above, we can compute the Chern connection of $(\cF, h_\cF)$ as follows. 
\begin{cor}\label{D1}
Let $u\in H^0(D,\cF)$ and let $\uu$ be a smooth representative of $u$. We have 
\[\nabla u = P (\mu) dt\]  
where 
\begin{enumerate}[label=-]
\item $\nabla$ is the $(1,0)$-part of the Chern connection on $(\mathcal F, h_\cF)$.
\item $\mu$ is defined by $D'\uu=dt\wedge \mu$, cf. \eqref{add3}, and $\mu|_{X_t}$ only depends on $u$. 
\item $P (\mu )$ is the fiberwise projection onto $H^0 (X_t, (K_{X_t} +L_t)\otimes \mathcal I(h_{L_t}))$ with respect to the $L^2$-norm. 
\end{enumerate}
\end{cor}

\begin{proof}
Let $\nabla$ be the $(1,0)$-part of the Chern connection of $(\cF, h_\cF)$. Then we have 
$$\nabla u = \sigma \otimes dt ,$$
where $\sigma= \frac{\nabla u}{dt} \in \cC^{\infty}(D, \cF)$. 
Let $u, v$ be two holomorphic sections of $\cF$ and let $f$ be a smooth function with compact support in $D$.
Since $v$ is a holomorphic, we have 
$$\int_{D} \langle u, v\rangle  i\partial\bar\partial f(t) = \int_{D} \langle \nabla u, v\rangle \wedge i \bar\partial f(t).$$
Let $\uu$ and  $\vv$ be the representatives of $u$ and $v$ respectively given by \eqref{uu}. The argument already used in Lemma \ref{IPP} shows that we have
$$
\int_{\cX^\circ} \uu\wedge \bar\vv e^{- \phi_L } i\partial\bar\partial f(t)  = \int_{\cX^\circ} 
D'\uu  \wedge \bar\vv e^{-\phi_L} i\bar\partial f(t)  .$$ 
Here $D'$ is, as before, the Chern connection on $(L\to \cX^\circ, h_L)$. As a consequence, we have 	
$$\int_D \langle \nabla u, v\rangle \wedge i \bar\partial f(t) = \int_{\cX^\circ} 
D'\uu  \wedge \bar\vv e^{-\phi_L} i\bar\partial f(t) .$$
Since we can choose $f$ on the base $D$ arbitrarily, we infer
\begin{align*}
\int_{X^\circ _t}\langle\sigma_t , v_t \rangle  \underset{\rm def}{=} \int_{X^\circ _t}\langle \frac{\nabla u}{dt}, v\rangle_t&= \int_{X^\circ _t} \frac{D' \uu}{dt}  \wedge \bar\vv e^{-\phi_L} \\
&= \int_{X^\circ _t} \mu \wedge \bar v_t e^{-\phi_L}\\
&= \int_{X^\circ _t} P(\mu|_{X_t}) \wedge \bar v_t e^{-\phi_L}
\end{align*}
for each $t\in D$.

\noindent As the above holds for any holomorphic section $v$, we obtain thus 
\begin{equation*}
\nabla u = P (\mu ) dt
\end{equation*} 
on $D$.
\end{proof}	
\medskip

\noindent We can now complete the proof of Proposition~\ref{curvv}.
\label{page curvv}
\begin{proof}[Proof of Proposition~\ref{curvv}] 
Let $f \in \cC_c ^{\infty} (D)$. By \eqref{degreecon}, we have 
$$ \int_{\cX^\circ} \psi_\ep \uu\wedge \overline{\bar\partial \uu} e^{-\phi_L} \bar\partial f(t) =0$$
for every $\ep$.
By integration by parts,  we obtain
\begin{equation}\label{eqqqq}
\int_{\cX^\circ}  \dbar \psi_\ep \wedge  \uu\wedge \overline{\bar\partial \uu} e^{-\phi_L}  f (t)
+ \int_{\cX^\circ}\psi_\ep \bar\partial\uu\wedge \overline{\bar\partial \uu} e^{-\phi_L}  f(t) + (-1)^n \int_{\cX^\circ}\psi_\ep  \uu\wedge  \overline{D' \bar\partial \uu} e^{-\phi_L} f(t) =0 .
\end{equation}

For the first term of \eqref{eqqqq}, by integration by parts again, we have 
{\small
\begin{align*}
(-1)^n \int_{\cX^\circ}  \dbar \psi_\ep \wedge  \uu\wedge \overline{\bar\partial \uu} e^{-\phi_L}  f (t) =& -\int_{\cX^\circ}  \dbar \psi_\ep \wedge  D' \uu\wedge \overline{\uu} e^{-\phi_L}  f (t) + \int_{\cX^\circ}  \partial\dbar \psi_\ep \wedge  \uu\wedge \overline{\uu} e^{-\phi_L}  f (t)\\
& - \int_{\cX^\circ}  \dbar \psi_\ep \wedge  \uu\wedge \overline{\uu} e^{-\phi_L} \wedge \partial f (t).
\end{align*}
}
Recall that $d\psi_\ep$ and $dd^c\psi_\ep$ are uniformly bounded with respect to $\om_E$ and converge to zero pointwise. Since $\uu$ and $D'\uu$ are $L^2$ by assumption, we see from Lebesgue dominated convergence theorem that the RHS tends to $0$. Therefore the first term of \eqref{eqqqq} tends to $0$.

 Since $\bar\partial\uu\wedge \overline{\bar\partial \uu} $ is $L^1$, the second term of \eqref{eqqqq} tends to 
$\int_{\cX^\circ} \bar\partial\uu\wedge \overline{\bar\partial \uu} e^{-\phi_L}  f(t) $.
We obtain thus
\begin{equation}\label{partialimi}
\int_{\cX^\circ} \bar\partial\uu\wedge \overline{\bar\partial \uu} e^{-\phi_L}  f(t) = (-1)^{n-1}\lim_{\ep \rightarrow 0}  \int_{\cX^\circ} \psi_\ep \uu\wedge  \overline{D' \bar\partial \uu} e^{-\phi_L} f(t) . 
\end{equation}
\medskip

\noindent We complete in what follows the proof of the proposition. We have
{\small
\begin{align*}
\int_{\cX^\circ} \uu\wedge \bar\uu e^{- \phi_L} \wedge \dbar \partial f(t) 
=& \lim_{\ep \rightarrow 0}\int_{\cX^\circ} \psi_\ep \uu\wedge \bar\uu e^{- \phi_L} \wedge\dbar \partial  f(t) \\
=& - \lim_{\ep \rightarrow 0}\Big[\int_{\cX^\circ} \dbar\psi_\ep\wedge \uu\wedge \bar\uu e^{- \phi_L} \wedge \partial f(t) + 
\int_{\cX^\circ} \psi_\ep\wedge \dbar \uu\wedge \bar\uu e^{- \phi_L} \wedge \partial f(t) \\
&+ (-1)^n \int_{\cX^\circ} \psi_\ep\wedge \uu\wedge \overline{D'\uu} e^{- \phi_L} \wedge \partial f(t) \Big]
\end{align*}}
Note that the first term tends to $0$ since $\uu$ is $L^2$. The second term vanishes because of \eqref{degreecon}. Then we have
$$\int_{\cX^\circ} \uu\wedge \bar\uu e^{- \phi_L} \wedge \dbar \partial  f(t) 
= (-1)^{n-1} \lim_{\ep \rightarrow 0} \int_{\cX^\circ} \psi_\ep\wedge \uu\wedge \overline{D'\uu} e^{- \phi_L} \wedge \partial f(t) . $$ 
Applying again integration by part, the RHS above becomes
{\footnotesize
\begin{equation}\label{eq001}
 \lim_{\ep \rightarrow 0} (-1)^{n-1}\int_{\cX^\circ} \partial\psi_\ep \wedge \uu\wedge \overline{D'\uu} e^{- \phi_L}  f(t)  + (-1)^{n-1}
\int_{\cX^\circ} \psi_\ep  D' \uu\wedge \overline{D'\uu} e^{- \phi_L}  f(t)
- \int_{\cX^\circ} \psi_\ep  \uu\wedge \overline{\dbar D'\uu} e^{- \phi_L}  f(t) .
\end{equation}}
\medskip

As $\uu$ and $D' \uu$ are $L^2$, the first term of \eqref{eq001} tends to $0$, and the second term of \eqref{eq001} tends to 
$\int_{\cX^\circ} D' \uu\wedge \overline{D'\uu} e^{- \phi_L}  f(t)$.
For the third term, as $\dbar D'\uu= \Theta_{h_L} (L) - D' \dbar \uu$, we have 
\begin{equation}\label{eq000}
\int_{\cX^\circ} \psi_\ep  \uu\wedge \overline{\dbar D'\uu} e^{- \phi_L}  f(t)  
= -\int_{\cX^\circ} \psi_\ep  \uu\wedge \overline{D' \dbar\uu} e^{- \phi_L}  f(t) -
\int_{\cX^\circ} \psi_\ep  \Theta_{h_L} (L)\uu\wedge \overline{\uu} e^{- \phi_L}  f(t) .
\end{equation}
Combining with \eqref{partialimi},  we obtain
$$\lim_{\ep \rightarrow 0} \int_{\cX^\circ} \psi_\ep  \uu\wedge \overline{\dbar D'\uu} e^{- \phi_L}  f(t) 
= (-1)^n \int_{\cX^\circ} \bar\partial\uu\wedge \overline{\bar\partial \uu} e^{-\phi_L}  f(t)  -
\int_{\cX^\circ} \Theta_{h_L} (L)\wedge \uu\wedge \overline{\uu} e^{- \phi_L}  f(t) .$$

All three terms of the RHS of \eqref{eq001} have now been calculated and the sum is just 
{\small
\[(-1)^{n-1} \int_{\cX^\circ}  D' \uu\wedge \overline{D'\uu} e^{- \phi_L}  f(t) + (-1)^{n-1}
\int_{\cX^\circ} \bar\partial\uu\wedge \overline{\bar\partial \uu} e^{-\phi_L}  f(t)  + 
\int_{\cX^\circ} \Theta_{h_L} (L)\wedge \uu\wedge \overline{\uu} e^{- \phi_L}  f(t) .\]
}
The proposition is thus proved.
\end{proof}

\subsection{A characterization of flat sections}
Now for applications, we need to generalize \cite[Prop 4.2]{Bo09} and formula  
\cite[(4.8)]{Bo09} to our singular setting. 
Following the argument of \cite[Prop 4.2]{Bo09}, we have the following.

\begin{proposition}\label{bobprop} We assume that the coefficients $b_I$ in the Set-up condition \ref{A2} are equal to zero.
Let $u$ be a holomorphic section of $\cF$ on $D$ such that $\nabla u(0) =0$. Then $u$ can be represented by a smooth $(n,0)$-form $\uu$ on $\cX^\circ$, $L^2$ with respect to $h_L , \omega_E$,  
such that
$$\dbar \uu= dt \wedge \eta$$ 
for some $L^2$-form $\eta$ which is primitive (with respect to $\omega_E |_{X^\circ}$)  on $X^\circ$, and 
$$D' \uu=d t \wedge \mu$$
for some $\mu$ satisfying  $\mu |_{X^\circ} =0$. 
Here $X^\circ := X_0 \cap \cX^\circ$ is on the central fiber.
\end{proposition}	

\begin{proof}
Let $\uu$ be the representative constructed in \eqref{uu}. We have 
$$D' \uu= dt \wedge \mu \qquad \dbar \uu=dt \wedge \eta .$$
Then $\mu|_{X^\circ}$ is orthogonal to the space of $L^2$-holomorphic section by Corollary~\ref{D1}.
\smallskip

By Remark \ref{prim} our representative $\uu$ has the following property 
\begin{equation}\label{defu1}
\uu\wedge \omega_E =dt \wedge d\ol t\wedge u_1 
\end{equation}
for some $(n,0)$-form $u_1$ on  $\cX^\circ$. It follows that we have 
\begin{equation}\label{defu12}
\eta \wedge \omega_E = 0
\end{equation}
on each fiber $X_t$.
\medskip


Moreover, as $\mu |_{X^\circ}$ is orthogonal to $\ker \dbar$, Theorem~\ref{Hodge} shows that 
$\mu|_{X^\circ}$ is $\dbar^{\star_0}$-exact, i.e., there exists a 
$\dbar$-closed $L^2$-form $\beta_0$ on $X^\circ$ such that 
$$\dbar^{\star_0}\beta_0  =\mu |_{X^\circ}.$$
Let $\widetilde \beta_0$ be an arbitrary (globally $L^2$) extension of $ \star_0 \beta_0$. Then
$\uu-dt  \wedge \widetilde \beta_0$ is the representative we are looking for.
\end{proof}
\medskip

The result above produces a representative enjoying nice properties in restriction to the central fiber. In order to generalize that to each fiber, we consider the case where $u \in H^0 (D, \cF)$ is a {\it flat} section with respect to $h_\cF$. For that purpose, we introduce an additional cohomological assumption. 

\medskip
In the Set-up~\ref{altsetup}, assume that the coefficients $b_I$ appearing in \ref{A2} vanish. That is to say, $h_L$ has analytic singularities in the usual sense. We let $\Delta_t ^{''}$ be the Laplace operator on $L^2$-integrable $(n,1)$-forms with values in $L$ 
on $X^\circ _t$, taken with respect to $\omega_E, h_L$. Let us consider the following assumption. 
\begin{enumerate}[label={\color{violet} (A.\arabic*)}]
\setcounter{enumi}{3}
\item \label{A4} The dimension $\dim \ker \Delta_t ^{''}$ is independent of $t\in D$. 
\end{enumerate}

\medskip
Note that by Bochner formula, we already know that $\dim \ker \Delta_t ^{''}<+\infty$. Indeed, if  $\alpha$ is a $\Delta_t ^{''}$-harmonic $L^2$ $(n,1)$-form on $X^\circ_t$, then \eqref{bochner2} shows that $\alpha$ is $\Delta_t ^{'}$-harmonic and $(D')^* \alpha=0$.
In particular, $\star \alpha$ is a $L^2$-holomorphic section. Thanks to Remark~\ref{extension}, we get an injection $\ker \Delta_t ^{''}\hookrightarrow H^0(X_t, \Omega^{n-1}_{X_t}\otimes L_t)$. In particular, the former space is finite-dimensional.



\medskip

\noindent The main result of this subsection states as follows. 
\begin{thm}\label{mainresI}
In the Set-up~\ref{altsetup}, assume that $h_L$ has analytic singularities and that the condition \ref{A4} above is satisfied. 

Let $u \in H^0 (D, \cF)$ be a flat section with respect to $h_\cF$. Then, we can find a continuous $(n,0)$-form $\uu$ on $\cX \setminus E$ 
representing $u$ such that  
\begin{enumerate}[label=$(\roman*)$]
\item $\uu$ is $L^2$ and $D' \uu=0$, 

\item $\eta |_{X^\circ _t}= 0$ for any $t\in D$, and $\dbar \uu\wedge \overline{\dbar \uu}= 0$, where $\eta$ is --as usual-- given by 
$\dbar \uu= dt \wedge \eta$. Moreover, the equality
\begin{equation}\label{twoequs}
\Theta_{h_L} (L) \wedge \uu =0
\end{equation}
holds true point-wise on $\cX\setminus E$. 
\end{enumerate}
\end{thm}

\begin{remark}
Let us collect a few remarks about the theorem. 
\begin{enumerate}[label=(\alph*)]

\item The content of Theorem \ref{mainresI} is clear: it "converts" the abstract data 
$\dbar u= 0$ and $\nabla u= 0$ into an effective result.

\item The identity \eqref{twoequs} is equivalent to saying that the hermitian metric induced by $i\Theta_{h_L}(L)$ on $\Lambda^nT^*_{\cX^\circ}$ has $\uu$ in its kernel.
\end{enumerate}
\end{remark}

\begin{proof}
As in the proof of  Proposition \ref{bobprop}, we start with a representative $\uu$ given by \eqref{uu}
(i.e. constructed via the contraction with the canonical lifting of $\displaystyle \frac{\partial}{\partial t}$ with respect $\om_E$). 

Since $u$ is flat on $D$, we have 
$D' \uu = dt \wedge \mu$ where $\mu |_{X^\circ _t}$ is $L^2$ and $\dbar^{\star_t}$-exact for every $t\in D$. Therefore we can solve the $\dbar^*$-equation fiberwise, namely there exists a unique $L^2$-form $\beta_t$ on $X^\circ _t$ such that $\beta_t$ is orthogonal to the $\ker \dbar^{\star_t}$ and such that
 $$\dbar^{\star_t}\beta_t = \mu |_{X^\circ _t}.$$
 By taking the $\dbar$ in both side and taking into account the fact that $\beta_t$ is orthogonal to the $\ker \dbar^{\star_t}$, we obtain
 \begin{equation}\label{lapmin}
 \Delta_t ^{''} \beta_t = \dbar( \mu|_{X_t^\circ} )\,\,\text{ on } X^\circ_t \qquad\text{and}\qquad \beta_t \text{ }\bot \text{ }\ker \Delta_t ^{''}.
  \end{equation}
 
 \noindent By analogy to the compact case it is expected that the minimal solution of a $\Delta_t ^{''}$ equation varies smoothly provided that $\dim \ker \Delta_t ^{''}$ is constant. We partly confirm this 
 expectation in Proposition~\ref{conlemme} by showing that it is continuous; for the moment, we will admit this fact
 and finish the proof of the theorem.
 
 \noindent We set 
 $$\uu_1 := \uu - dt \wedge (\star_t \beta_t ).$$
 It is a continuous, fiberwise smooth form on $\cX^\circ$ and we show now that $\uu_1$ is a representative for which the points (i)-(ii) above are satisfied. 
 
 \medskip
 
\noindent By construction, $D' \uu_1 =0$ on $\cX^\circ$.  By \eqref{e6} of Proposition~\ref{conlemme} below, 
 the $L^2$-norm of $\beta_t$ is smaller than the $L^2$-norm of $\dbar \mu |_{X^\circ _t}$, i.e., 
 $$\|\beta_t\|_{L^2} \leq C \| \dbar\mu |_{X^\circ _t}\|_{L^2}$$
 for some constant $C$ independent of $t$. 
Moreover, we recall the estimates in Lemma \ref{meta}: $\dbar\mu |_{X^\circ _t}$ is uniformly $L^2$-bounded. Therefore $dt \wedge (\star_t \beta_t )$ is $L^2$ and so our representative $\uu_1$ is $L^2$. 
 
 \smallskip
 
 \noindent We have $\dbar \uu_1= dt\wedge \left(\eta+ \dbar \big(\star_t\beta_t\big)\right)$
 and since $\dbar \big(\star_t\beta_t\big)\wedge \om_E= \dbar \beta_t= 0$, it follows that 
  \begin{equation}\label{x1}
  \frac{\dbar \uu_1}{dt}\Big|_{X_t}\wedge \om_E= 0.
 \end{equation} 

In order to use Proposition \ref{curvv}, we show next that we have $\dbar \uu_1\wedge \ol{\dbar \uu_1}\in L^1$. To this end, we write 
 \begin{align}
 \dbar \uu_1 \wedge \overline{\dbar \uu_1} = & \label{l1} dt \wedge \eta  \wedge \overline{dt \wedge \eta}  + dt \wedge \dbar  (\star_t \beta_t )  \wedge\overline{dt \wedge \eta} \\ 
+ & d\label{l2} t \wedge \eta \wedge \overline{dt \wedge \dbar  (\star_t \beta_t )}
 + dt \wedge \dbar  (\star_t \beta_t ) \wedge \overline{dt \wedge \dbar  (\star_t \beta_t )}.\\
\nonumber
\end{align}
 By the estimates in Lemma  \ref{meta}, $\eta$ is $L^2$. Then the first term of RHS of \eqref{l1} is $L^1$. Degree considerations show that we have
 \begin{equation}\label{degconn}
 dt \wedge \dbar  (\star_t \beta_t )  \wedge\overline{dt \wedge \eta} = dt \wedge \dbar _t (\star_t \beta_t )  \wedge\overline{dt \wedge \eta} ,
  \end{equation}
 where $\dbar _t$ is the $\dbar$-operator on $X_t$. Since $\Delta_t ^{''} \beta_t = \dbar \mu $ and $\beta_t$ is of degree $(n,1)$, Bochner formula shows that the $L^2$-norm of $\dbar _t (\star_t \beta_t)$ is equal to the $L^2$ norm of the form $(D')^{\star_t} \beta_t $ (this is due to the fact that $\beta_t$ is $\dbar$-closed), which in turn is bounded by the $L^2$-norm of $\dbar \mu$. Once again,  
the estimates provided by Lemma \ref{meta} show that the $L^2$-norm of $\dbar \mu  |_{X_t} $ is bounded uniformly with respect to $t$. It follows that $dt \wedge \dbar _t (\star_t \beta_t ) $ is $L^2$.
 
 Therefore $dt \wedge \dbar  (\star_t \beta_t )  \wedge\overline{dt \wedge \eta}$ is $L^1$-bounded by using \eqref{degconn}. The same type of arguments show that the two terms in \eqref{l2} are also $L^1$.
\medskip

\noindent We apply Proposition \ref{curvv} for the representative $\uu_1$ of $u$. The flatness of $u$ imply that  
$$
c_n\int_{\cX^\circ} ((-1)^n\dbar \uu_1\wedge \overline{\dbar \uu_1} + i\Theta_{h_L} (L) \wedge \uu_1\wedge \overline{\uu_1}) e^{-\phi_L}  =0.$$
Thanks to the assumption $ i\Theta_{h_L} (L) \geq 0$ combined with \eqref{x1}, both two terms in the integral are semi-positive, and modulo the continuity of the family $(\beta_t)$ our result is proved.
\end{proof}	
\medskip

\noindent The proof of the following result is very similar to the familiar situation in which the 
couple of metrics $(\om_E, h_L)$ are non-singular. We provide a complete argument because we were unable to find a reference.

\begin{proposition}\label{conlemme}
The minimal solution $\beta_t$ in \eqref{lapmin} varies continuously with respect to $t$.
\end{proposition}

\begin{proof} We have divided our proof in a few steps. 
	
	\subsubsection{Step 1} Let $(u_t)$ be a family of $L$-valued, $L^2$ forms of $(n,1)$--type on the fibers of $p$, such that we have 
	\begin{equation}\label{e5}
	\Delta_t''v_t= u_t
	\end{equation}
	on the fiber $X_t$. If moreover we assume that each $v_t$ is perpendicular to $\ker\Delta_t''$, then we claim that 
	\begin{equation}\label{e6}
	\int_{X_t}|v_t|^2_{\om_E}e^{-\phi_L}dV_{\om_E}\leq C \int_{X_t}|u_t|^2_{\om_E}e^{-\phi_L}dV_{\om_E}
	\end{equation}
	for some constant $C$ uniform with respect to $t$.
	
	\noindent Indeed, our claim follows instantly from Corollary \ref{unifconst} and \eqref{hp} applied to $u:= \star_t v_t$.
	
	\medskip
	
	\subsubsection{Step 2} Let $\lambda\in \mathbb{C}$ such that $0< |\lambda|\ll 1$ --we will make this precise in a moment.   We claim that the operator $$A_{\lambda, t}:=\lambda -\Delta^{''} _t$$ is invertible, which we show by proving that the equation $A_{\lambda,t} v = u$ admits a solution $v$, as soon as $u$ is in $L^2$. This can be seen via the usual Riesz representation theorem, as follows. 
	
	\noindent
	 We define on $L^2_{n,1}$ the functional
	\begin{equation*}
	I(\phi)= \int_{X_t}\langle u, \phi \rangle e^{-\phi_L}dV_{\om_E}
	\end{equation*}
	We write $u= u_1+ u_2$ and $\phi= \phi_1+ \phi_2$ according to the decomposition $L^2_{n,1}= \ker \Delta_t''\oplus (\ker \Delta_t'')^\perp$. Then we have 
	\begin{equation*} 
	I(\phi)= \int_{X_t}\langle u_1, \phi_1 \rangle e^{-\phi_L}dV_{\om_E}+ \int_{X_t}\langle u_2, \phi_2 \rangle e^{-\phi_L}dV_{\om_E}
	\end{equation*}
	and then the squared absolute value of the second integral is smaller than
	\begin{equation}\label{d3}
	\int_{X_t}|\Delta''_t\phi |^2e^{-\phi_L}dV_{\om_E}
	\end{equation}
	up to a uniform constant, by Step 1. Therefore, we get
	\begin{equation}\label{d4}
	|I(\phi)|^2\lesssim \int_{X_t}|\phi_1|^2e^{-\phi_L}dV_{\om_E}+ \int_{X_t}|\Delta''_t\phi |^2e^{-\phi_L}dV_{\om_E}.
	\end{equation}
	Since \[\Big|\int_{X_t} \langle\phi, \Delta_t''\phi \rangle dV_{\om_E} \Big|^2\le  \int_{X_t}|\phi_2|^2e^{-\phi_L}dV_{\om_E}\cdot  \int_{X_t}|\Delta''_t\phi|^2e^{-\phi_L}dV_{\om_E}\]
	we see that from \eqref{d4} that
	\begin{equation}
	\label{d5}
	|I(\phi)|^2\le C_{u,\lambda} \int_{X_t}|\lambda\phi- \Delta''_t\phi |^2e^{-\phi_L}dV_{\om_E}
	\end{equation}
	as we see from the previous step, provided that $\displaystyle |\lambda|\leq \frac{1}{2C},$
	where $C$ is the constant in \eqref{e6}. Moreover, $C_{u,\lambda}$ is of the form $ C_{\lambda} \cdot \|u\|_{L^2}^2$. 
	\smallskip
	
	\noindent 
	
 Taking $\phi=u$ in the identity above, we see that $A_{\lambda,t} :=\lambda-\Delta_t''$ is injective. Moreover, the functional 
	\begin{eqnarray*}
	J: \Imm A_{\lambda,t} & \longrightarrow &\mathbb C \\
	A_{\lambda,t}\phi & \mapsto& I(\phi)
	\end{eqnarray*}
	 is well-defined and continuous by \eqref{d5}. In particular, it extends to $ F=\overline{\Imm A_{\lambda,t}}$ and Riesz theorem provides us with an element $v\in F$ satisfying 
	 \begin{equation}
	 \label{riesz}
	 \forall \psi\in F, \,\,J(\psi)=\int_{X_t}\langle v,\psi\rangle e^{-\phi_L}dV_{\om_E} \qquad  \mbox{and} \qquad  \|v\|_{L^2} \le C_{\lambda} \|u\|^2_{L^2}.
	 \end{equation}
	  The equality $J( A_{\lambda,t} \phi)=I(\phi)$ for any $\phi$ in $L^2$ shows that $A_{\lambda,t} v=u$. This concludes this step.

%
%
%
		
	\medskip 
	
	\subsubsection{Step 3} Let $\lambda\in \mathbb{C}$ as in the previous step,
	and let $u_t$ be a continuous $L^2$-family. We show that $v_t :=(\lambda -\Delta^{''}_t)^{-1} u_t$ is continuous with respect to $t$ (with respect to the $L^2$-norm). 
	
	\noindent It would be sufficient to check the continuity at one point $0\in D$. 
	For any $\ep >0$ we define the form
	$v_{0,\ep}:= \mu_\ep v_0$ with compact support in $X_0\setminus E$. We then have
	$$\|v_{0,\ep} -v_0\|_{L^2} \leq \ep ,\qquad \| (\lambda -\Delta^{''}_0) v_{0,\ep} - (\lambda -\Delta^{''}_0) v_0 \|_{L^2} \leq \ep$$
by the properties of $(\mu_\ep)_{\ep> 0}$.	

We next construct a smooth extension $v_\ep$ of $v_{0, \ep}$ as follows. Let $(\Omega_i)_{i\in I}$
be a finite covering of $p^{-1}(\frac 12 D)$ by coordinate charts, and let $(\theta_i)_{i\in I}$ be a partition of unity subordinate to this covering. The $L$-valued form
$$v_\ep:=  \mu_\ep\sum_i\theta_i(z, t)v_{0,i}$$
extends $v_{0, \ep}$ and it is compactly supported in $\cX \setminus E$. Here we denote by $v_{0,i}$ is the local expression of $\displaystyle v_{0,\ep}|_{X_0\cap \Omega_i}$, extended trivially to $\Omega_i$ (note that this is still $L^2$).

	Since the metrics $\omega_E, h_L$ are smooth in $\cX \setminus E$, 
	$$u_{\ep, t} := (\lambda -\Delta^{''}_t) (v_\ep |_{X_t})$$ 
	is a smooth $L^2$-family. 
	
	\noindent By the second part in \eqref{riesz} we have 
	$$ \| v_t -  v_\ep |_{X_t}\|_{L^2} \leq C \| u_t -  u_{\ep, t}\|_{L^2} . $$
	As $u_t$ and $u_{\ep, t}$ are continuous with respect to $t$, we infer that we have $\| v_t -  v_\ep |_{X_t}\|_{L^2}  \leq C_\ep\mathcal O(|t|) + C \| u_0 -  u_{\ep, 0}\|_{L^2} $.
	It follows that we have 
	$$\| v_t -  v_\ep |_{X_t}\|_{L^2} \leq o(1) + C \ep $$
as $|t|\to 0$. The -small- quantity $o(1)$ here depends on $\ep$, but since
by construction the family $v_\ep |_t$ is continuous with respect to $t$
and its continuity modulus is independent of $\ep$ we infer that $v_t$ is continuous at $0$.
	
	\smallskip
	
	\noindent Now we define the operator $\Pr_t :=\int_{\lambda \in \Gamma} (\lambda  -\Delta^{''} _t)^{-1} d\lambda$ where $\Gamma$ is a small circle centered at $0$. We have proved above 
	that $\Pr_t$ is continuous with respect to $t$. Moreover, $\Pr_t$ coincides with the orthogonal projection onto $\ker \Delta^{''} _t$: we postpone the proof of this claim for the moment, see the Remark \ref{project} below.
	
\medskip 
	
\subsubsection{Step 4} This is the main step in the proof of the proposition. Let $\beta_t$ be the $\dbar^*$-solution on $X_t$ in question. Then
	$$\Delta^{''} _t \beta_t= \dbar (\mu|_{X_t^\circ}) .$$
	By the estimates in Proposition \ref{bobprop}, the RHS is $L^2$ and continuous with respect to $t$.
	Let $s_t$ be a continuous $L^2$-family (continuous with respect to $t$) such that $s_0 = \beta_0$.
	Then we have
	$$\lambda_0 s_t - \Delta'' _t \beta_t =\lambda_0 s_t - \dbar (\mu |_{X_t^\circ}) $$
	for every $t$, where $0< |\lambda_0|\ll 1$ is fixed, small enough as in Step 2.
	By Step 3, we can find a continuous family $\gamma_t$ such that
	$$\lambda_0 s_t - \Delta^{''} _t \beta_t  = \lambda_0 \gamma_t - \Delta^{''} _t \gamma_t $$
	for every $t$.
	Then $\Delta^{''} _t (\beta_t - \gamma_t^\perp) =\lambda_0 (s_t -\gamma_t )$, 
	where $\gamma_t^\perp:= \gamma_t- \Pr_t\gamma_t$ is the projection onto $(\ker \Delta''_t)^{\perp}$.
	
	\noindent Now $\beta_t$ is orthogonal to $\ker \Delta^{''} _t $ by construction.  Then 
	$\beta_t - \gamma_t^\perp$ is  orthogonal to $\ker \Delta^{''} _t $. By Step 1 we thus have 
	$$\|\beta_t - \gamma_t^\perp\|_{L^2} \leq C  \| s_t -\gamma_t \|_{L^2} .$$
	Note that $s_0 = \gamma_0$ and $s_t$ and $\gamma_t$ are continuous, then 
	$\|\beta_t - \gamma_t^\perp \|_{L^2}  = o (1)$. By Step 4, $\gamma_t^\perp$ is continuous, therefore $\beta_t$ is continuous at $0$. 
	
	Summing up, the continuity with respect to the $L^2$ norm of 
	$(\beta_t)_{t\in D}$ is established.	
\medskip 
	
\subsubsection{Step 5} We show here that the form
\begin{equation}
dt\wedge \star_t\beta_t
\end{equation}
induced by the family $(\beta_t)_{t\in D}$ in \eqref{lapmin} is continuous on $\cX\setminus E$. This 
is a consequence of the fact that the family of operators $(\Delta''_t)_{t\in D}$ is smooth and it has a smooth variation when restricted to a compact subset $K\subset \cX\setminus E$, combined with the continuity property established in Proposition~\ref{conlemme}.
\smallskip

\noindent Let $\Omega\Subset X_0\setminus E$ be a small coordinate chart. We can interpret the $(\Delta''_t)_{t\in D}$ as family of operators on the forms defined on $\Omega$, since $p$ is locally trivial. Then we have 
\begin{equation}
\Delta_0''\beta_t= \dbar \mu_t+ (\Delta_0''- \Delta_t'')(\beta_t)
\end{equation}
from which it follows that 
\begin{equation}\label{E2}
\Delta_0''(\beta_t- \beta_0)= \dbar \mu_t- \dbar \mu_0+ (\Delta_0''- \Delta_t'')(\beta_t).
\end{equation} 
The equality \eqref{E2} combined with the usual a-priori estimates for the elliptic operators imply that
\begin{equation}\label{E3}
\Vert \beta_t- \beta_0\Vert^2_{W^2}\leq C\big(\Vert \dbar \mu_t- \dbar \mu_0\Vert^2_{L^2}+  \Vert \beta_t- \beta_0\Vert^2_{L^2}\big)+ 
\delta_t \Vert \beta_t\Vert^2_{W^2}
\end{equation}
where $\delta_t\to 0$ as $t\to 0$. We infer that 
\begin{equation}\label{E4}
\Vert \beta_t- \beta_0\Vert^2_{W^2}\leq \delta_t
\end{equation}
for some (other)  function $\delta_t$ tending to zero.

\noindent The usual boot-strapping method implies that $\displaystyle \lim_{t\to 0}\beta_t= \beta_0$  smoothly on any compact subset in  $\cX \setminus E$.
In global terms this translates as 
$$dt\wedge \star_t\beta_t$$
is a continuous $(n, 0)$-form on $\cX \setminus E$, so our lemma is proved.
\end{proof}
\medskip

\begin{remark}\label{project}
	{\rm  For the sake of completeness, we provide the details for the fact that the linear operator $\Pr_t :=\int_{\lambda \in \Gamma} (\lambda  -\Delta^{''} _t)^{-1} d\lambda$ is the orthogonal projection onto $\ker\Delta_t''$. Let $\mathcal H_t$ be the Hilbert space $(\ker\Delta_t'')^\perp$. We need to prove two points:
		
\begin{enumerate}[label=$(\roman*)$]
\item   $\Pr_t u=0$ for any $u\in \mathcal H_t$;
		
\item  $\Pr_t u=u$ for any $u\in \ker \Delta'' _t$.
\end{enumerate}	
		For the first point,  we have the following equality
		\begin{equation}\label{d7}
		(\lambda  -\Delta''_t)^{-1}= -\sum_{k\geq 0}\lambda^k{\cG}_t^{k+1}   \qquad\text{on } \mathcal H_t
		\end{equation}
		where $\cG_t$ is the inverse of the operator $\Delta_t''$ restricted to $\mathcal H_t$. The sum in \eqref{d7} is indeed convergent (for the operator norm), given the estimates \eqref{e6} and the fact that $\lambda$ belongs to the circle $\Gamma$ of small enough radius. 
		
		\noindent It follows that we can exchange integration/sum and then we have 
		\begin{equation}\label{d8}
		\int_{\lambda\in \Gamma}(\lambda  -\Delta''_t)^{-1}ud\lambda= -\sum_{k\geq 0}{\cG}_t^{k+1}(u)\int_{\lambda\in \Gamma}\lambda^kd\lambda   
		\end{equation}
		and this shows that $
		\Pr_t(u)= 0 $
		for any $u\in \mathcal H_t$. 
		
		\medskip
		
		For the second point, let $u\in \ker \Delta'' _t$. For  any $\alpha\in \ker \Delta_t''$, we have $$\langle \frac{u}{\lambda} , \alpha\rangle =\langle \frac{(\lambda -\Delta''_t) 	(\lambda  -\Delta''_t)^{-1} u}{\lambda} , \alpha\rangle = \langle (\lambda  -\Delta''_t)^{-1} u, \alpha\rangle  .$$ 
		Therefore
		$$\langle  u , \alpha\rangle = \int_{\lambda \in \Gamma}  \langle \frac{u}{\lambda} , \alpha\rangle d\lambda = \int_{\lambda \in \Gamma} \langle  (\lambda  -\Delta''_t)^{-1} u, \alpha\rangle  d\lambda.$$
		Then $u - \Pr_t u $ is orthogonal to $\ker \Delta^{''}$. 
		On the other hand, thanks to the equality 
		$$\Delta'' _t \circ 	(\lambda  -\Delta''_t)^{-1} (u)= 	(\lambda  -\Delta''_t)^{-1} \circ \Delta'' _t u =0 ,$$  
		we know that 
		$\Pr_t u \in \ker \Delta'' _t$.  Therefore $\Pr_t u  =u$.
		
}\end{remark}

\begin{remark}{\rm Actually the form $\beta_t$ can be obtained as usually via an integral formula,
\begin{equation}\label{i1}
\beta_t= -\int_{\lambda\in \Gamma}\frac{1}{\lambda}(\lambda  -\Delta''_t)^{-1} (\dbar \mu|_{X_t^\circ})
\end{equation}
which gives the hope that its variation with respect to $t$ is actually
smooth. This can probably be obtained along the
same lines as in \cite[Thm 7.5]{Kod} modulo the fact that in the present situation, we have to deal with the additional difficulty induced by the fact that we are working with singular metrics
$\om_E$ and $h_L$. 
}\end{remark}
\medskip

\noindent We can now end this section by providing a proof of Theorem~\ref{thmA}. 

\begin{proof}[Proof of Theorem~\ref{thmA}]
Up to shrinking $D$ to a punctured disk $D_1\subset D$, one may assume that the assumptions \ref{A1}-\ref{A3} are satisfied (with $b_I=0$ for each $I$), cf. Section~\ref{sec1}. 

Next, there exists another punctured disk $D_2\subset D_1$ such that the coherent sheaf $R^1p_*(K_{\cX/D}\otimes L \otimes \mathcal I(h_L))$ is locally free and commutes with base change; i.e. its fiber at $t\in D_2$ is given by $H^1(X_t, K_{X_t}\otimes L|_{X_t} \otimes  \mathcal I(h_L|_{X_t}))$ and the dimension of the latter is independent of $t\in D_2$. Thanks to Proposition~\ref{cohomology}, the dimension of the space of harmonic $(n,1)$-forms, i.e. $\dim \ker(\Delta''_t)$, is independent of $t\in D_2$. In other words, the condition~\ref{A4} is satisfied over $D_2$. 

Theorem~\ref{thmA} is now a direct consequence of Theorem~\ref{mainresI}. 
\end{proof}

\section{A lower bound for the curvature in case of a -relatively- big twist}
\label{section 5}
 Let $p:\cX\to D$ be a smooth, projective family, and let $L\to \cX$ be a line bundle endowed with a metric $h_L=e^{-\phi_L}$ satisfying the following requirements.

\begin{enumerate}[label={\color{violet}(B.\arabic*)}]

\item \label{B1} There exist a smooth, semi-positive real (1,1)-form $\omega_L$ as well as an effective $\mathbb R$-divisor $E_0$ 
on $\cX$ such that 
$$i\Theta_{h_L}(L)= \omega_L+ [E_0]$$
where we denote by $[E_0]$ the current of integration associated to the $\mathbb{R}$-divisor $E_0$.

\item \label{B2} $ \omega_L$ is relatively Kähler, i.e., $\om_L|_{X_t} >0$ for every $t$.

\item \label{B3} The support of the divisor  $E:=\mathrm{Supp}(E_0)$ is snc, and transverse to the fibers of $p$.
\end{enumerate}
\smallskip

\noindent 
Let \[\displaystyle c(\phi_L):= \frac{\omega_L^{n+1}}{\omega_L ^n\wedge idt\wedge d\ol t}\] be the so-called geodesic curvature associated to $\omega_L$.


Our goal here is to establish the following result.

\begin{thm}\label{strikt}
Under the assumptions \ref{B1}-\ref{B3} above, let $(\cF, h_\cF)$ be the direct image bundle $p_\star\big((K_{\cX/D}+ L)\otimes \mathcal I(h_L)\big)$ 
endowed with the $L^2$ metric. Then for every $u\in H^0(D, \cF)$ and every $t\in D$, the following inequality holds
\begin{equation}\label{c1}
\langle\Theta_{h_\cF}(\cF)u, u\rangle_t\geq c_n\int_{X_t}c(\phi_L)u\wedge \ol u e^{-\varphi_L}
\end{equation}
where we identify $\Theta_{h_\cF}(\cF)$ with an endomorphism of $\cF$ 
by "dividing" with $idt\wedge d\ol t$.
\end{thm}
\medskip

\noindent Prior to providing the arguments for Theorem \ref{strikt} we propose here the following problem. 

\begin{question}
We assume that $\mathcal Y$ is a foliation on a K\"ahler manifold $Z$. In which cases the bundle $K_{\mathcal Y}+ L$ admits a 
positively curved metric? That is to say, is there some analogue of the Bergman metric on twisted relative canonical bundles in the more general contexts of a foliation? If yes, can we equally obtain
a lower bound of the curvature form?
\end{question}

\begin{proof}[Proof of Theorem~\ref{strikt}] 

The idea of our proof is to construct an approximation of the metric $h_L$ so that the resulting absolutely continuous part of the associated curvature form has Poincar\'e singularities along 
the support of $E$. Then we can use the curvature formulas we have obtained in the previous sections, and finally conclude by a limit argument. 

\subsubsection{Approximation of the metric} Let $\ep> 0$ be a (small) positive real number. We introduce the form
\begin{equation}\label{c2}
\omega_{\ep}:= \omega_L - \ep \sum_{i\in I}dd^c\log \log \frac 1{|s_i|^2}
\end{equation}
where $I$ is the set of irreducible components of $E$, and $s_i$ cut outs exactly one of these for a given $i\in I$, following notation in Section~\ref{sec1}. We note that $\omega_{\ep}$ is positive and has Poincaré singularities along $E$ as soon as the metrics $h_i$ used to measure the norm of $s_i$ are suitably scaled, which is what we assume from now on. 

Next, we introduce the following weight on $L$

\begin{equation}
\label{phi ep}
\phi_\ep:=\phi_L- \ep \sum_{i\in I} \log \log \frac1{|s_i|^2}.
\end{equation}
Clearly, $\phi_\ep$ has generalized analytic singularities in the sense of \ref{A2} in Set-up~\ref{altsetup} and it satisfies $dd^c \phi_\ep=\om_\ep$.

\noindent The properties of $(\phi_\ep)_{\ep}$ are collected in the following statement.

\begin{lemme}\label{app} Let $(\Omega, (z_1,\dots, z_n, t=z_{n+1})$ be a coordinate system on $\cX$ adapted to the pair $(\cX, E)$ as in Set-up~\ref{altsetup}. Then the following hold.
\begin{enumerate}[label=(\roman*)]
\item The geodesic curvature $c(\phi_\ep)$ is uniformly bounded from above. 
\item We have $\displaystyle \lim_{\ep\to 0}c(\phi_\ep)= c(\phi_L)$ point-wise on $\cX\setminus E$.
\item For every $\ep$ small enough, the multiplier ideal sheaf of $h_\ep:=e^{-\phi_\ep}$ coincides with 
$\mathcal I(h_L)$. Moreover, the induced $L^2$ metric, say $H_\ep$ on the direct image is smooth, and it converges to $h_\cF$ as $\ep\to 0$.  
\end{enumerate}
\end{lemme} 

\begin{proof}
For the point $(i)$, we write $c(\om_\ep)=1/\|dt\|^2_{\om_\ep}$ and the result follows from the transversality conditions \ref{B3} and e.g. the estimates for the coefficients provided in \eqref{coeff poin poin}. 

The point $(ii)$ follows easily from the local smooth convergence $\phi_\ep \to \phi_L$ on $\cX\setminus E$ combined with the positivity requirement \ref{B2}. 

As for the third point $(iii)$, the smoothness of $H_\ep$ and its convergence to $h_\cF$ is a consequence of the transversality assumption \ref{B3} by the same arguments as for Lemma~\ref{prelemme}. As for the statement about multiplier ideal sheaves, one has clearly $\mathcal I(\phi_\ep)\subset \mathcal I(\phi_L)$ while the reverse inclusion is an easy consequence of \ref{B1}-\ref{B3}. \end{proof}


\subsubsection{Application of the curvature formula} 
We consider $u$ a local holomorphic section of the bundle $\cF$, and let $\uu_\ep$ be the representative of $u$ constructed in \eqref{uu}, by using the contraction with the vector field $V_\ep$ associated to 
the metric $\omega_\ep$. 

\noindent Let 
\begin{equation}\label{c3}
\dbar \uu_\ep= dt\wedge \eta_\ep, \qquad D'\uu_\ep= dt\wedge \mu_\ep
\end{equation} 
where $D'= D'_\ep$ is the Chern connection corresponding to $(L, h_\ep)$.
Moreover we have 
\[  \omega_\ep \wedge \uu_\ep \wedge  \ol \uu_\ep =c(\phi_\ep) \uu_\ep \wedge \ol \uu_\ep \wedge p^\star (dt \wedge d\overline{t}) \quad \mbox{on } \,\, \cX \setminus E\] 
by \cite[Lem 4.2]{Bo11}.  Proposition \ref{curvv} then gives
\begin{equation}\label{c4}
-\frac{\partial ^2}{\partial t\partial \ol t}(\Vert u\Vert_{H_\ep}^2)= c_n\int_{X_t}c(\phi_\ep)\uu_\ep \wedge \ol \uu_\ep e^{-\varphi_\ep}+ \int_{X_t}|\eta_\ep|^2e^{-\phi_\ep}dV_{\omega_\ep}- \int_{X_t} |\mu_\ep|^2e^{-\phi_\ep}dV_{\omega_\ep},
\end{equation}
since $\eta_\ep$ is primitive on fibers of $p$. We discuss next the terms which occur in \eqref{c4}.
\smallskip

\noindent $\bullet$ The LHS of \eqref{c4} is equal to  
\begin{equation}\label{c5}
\langle\Theta_{H_\ep}(\cF)u, u\rangle- \Vert P(\mu_\ep)\Vert^2
\end{equation}
by the usual formula of the Hessian of the norm of a holomorphic section of a vector bundle.
Then \eqref{c4} becomes 
\begin{equation}\label{c6}
\langle\Theta_{H_\ep}(\cF)u, u\rangle_t= c_n\int_{X_t}c(\phi_\ep)\uu\wedge \ol \uu e^{-\varphi_\ep}+ \int_{X_t}|\eta_\ep|^2_{\omega_\ep}e^{-\phi_\ep}dV_{\omega_\ep}- \int_{X_t} |\mu_\ep^\perp|^2_{\omega_\ep}e^{-\phi_\ep}dV_{\omega_\ep},
\end{equation}
where $\mu_\ep= P(\mu_\ep)+ \mu_\ep^\perp$ is the $L^2$ decomposition of $\mu_\ep$ according to the $\ker \dbar$ and its orthogonal.  
\smallskip

\noindent $\bullet$ As observed in \cite[Lem 4.4]{Bo11}, we have 
\begin{equation}\label{c7}
\dbar \mu_\ep= D'\eta_\ep
\end{equation}
--even if the curvature is not zero!--, and actually $\mu_\ep^\perp$ is the solution of \eqref{c7}
\emph{whose $L^2$ norm is minimal}. By \cite[Thm~1.6]{JCMP}, we have the precise estimate
\begin{equation}\label{c8}
\int_{X_t} |\mu_\ep^\perp|^2_{\omega_\ep}e^{-\phi_\ep}dV_{\omega_\ep}\leq \int_{X_t}|\eta_\ep|^2_{\omega_\ep}e^{-\phi_\ep}dV_{\omega_\ep}
\end{equation}
and then we get 
\begin{equation}\label{c9}
\langle\Theta_{H_\ep}(\cF)u, u\rangle_t\geq  c_n\int_{X_t}c(\phi_\ep)u\wedge \ol u e^{-\phi_\ep}.\end{equation}
as consequence of \eqref{c6}.
\smallskip

\noindent $\bullet$ The last step in our proof is to notice that as the parameter $\ep$ approaches zero, the inequality \eqref{c9} implies
\begin{equation}\label{c10}
\langle\Theta_{h_\cF}(\cF)u, u\rangle_t\geq  c_n\int_{X_t}c(\phi_L)u\wedge \ol u e^{-\phi_L}.\end{equation}
Indeed, we are using Lemma \ref{app} for the LHS of \eqref{c9} and Lemma \ref{app} (i) combined with 
dominated convergence theorem for the RHS. Theorem \ref{strikt} is proved.
\end{proof}

In the last lines, we now explain how to deduce Theorem~\ref{thmB} from Theorem~\ref{strikt} above. 

\begin{proof}[Proof of Theorem~\ref{thmB}]
\label{page thmB}

We start by making the observation that if $\pi:\cX'\to \cX$ is a proper birational morphism inducing birational morphisms $X_t'\to X_t$, then one has $\int_{X_t'}c(\phi_L') u'\wedge \bar u' e^{-\phi_L'}= \int_{X_t}c(\phi_L) u\wedge \bar ue^{-\phi_L}$, with the self-explanatory notation. 

Therefore, by blowing up $\cX$ and restricting the family to a punctured disk $D_1\subset D$, one can from now on assume that the conditions \ref{B1} and \ref{B3} are satisfied. 

\medskip

Now, one has to show that one can further assume that condition \ref{B2} is satisfied. This is a bit more involved and can be shown as follows. 

Since $\int_{X_t} \om_L^n>0$ and $\om_L$ is smooth, it follows from e.g. \cite{Bou02} that $[\om_L]$ is $p$-big. In particular, there exists a punctured disk $D_2\subset D_1$, an effective, horizontal $\mathbb R$-divisor $F$ and an ample $\mathbb R$-line bundle $A$ on $\cX$ such that 
\begin{equation}
\label{omL}
[\om_L]=A+F \quad \mbox{in } H^{1,1}(\cX, \mathbb R).
\end{equation} 
After blowing-up once again and restricting to a smaller punctured disk $D_3\subset D_2$, one can assume without loss of generality that $E+F$ is snc and transverse to the fiber. Of course, the pull-back of $A$ is not ample anymore, but there exists an effective divisor $G$ contained in the exceptional locus of the blow-up such that $A-G$ is ample. All in all, one will assume from now on that one has a decomposition \eqref{omL} where $A$ is ample and $E+F$ is snc and transverse to the fibers. 

We pick a strictly psh smooth weight $\phi_A$ on $A$ and set $\phi_E$ (resp. $\phi_F$) for the singular psh weight on the corresponding $\mathbb R$-divisor. 

For $\delta>0$, we introduce the psh weight $\phi_\delta$ on $L$ defined by 
\[\phi_\delta= (1-\delta)\phi_L+\delta(\phi_A+\phi_F+\phi_E).\]
Clearly, $\phi_\delta$ has analytic singularities along the divisor $E+F$ and $(dd^c\phi_\delta)_{\rm ac}$ is a relative Kähler metric for any $\delta>0$. That is, the metric $h_{L,\delta}:=e^{-\phi_\delta}$ satisfies \ref{B2}. 

Thanks to Theorem~\ref{strikt}, the proof of Theorem~\ref{thmB} will be complete once we show the following

\begin{claim}
\label{claim}
With the notation above, one has 
\begin{enumerate}[label=(\roman*)]
\item $\mathcal I(\phi_\delta)=\mathcal I(\phi_L)$ for $\delta$ small enough. 
\item The $L^2$ metric $H_\delta$ induced by $h_{L,\delta}$ on $\cF$ is smooth and converges smoothly to $h_\cF$ when $\delta\to 0$. 
\item For any $t\in D_3$ and $u\in \cF_t$, one has 
\[\lim_{\delta \to 0} \int_{X_t}c(\phi_\delta)u\wedge \bar ue^{-\phi_\delta}=\int_{X_t}c(\phi_L)u\wedge \bar ue^{-\phi_L}.\]
\end{enumerate}
\end{claim}

\begin{proof}[Proof of Claim~\ref{claim}]
Since $\phi_L-\phi_E$ it is smooth (its curvature is nothing but $\om_L$), we have $\mathcal I(\phi_L)=\mathcal I(\phi_E)$ and $\mathcal I(\phi_\delta)= \mathcal I(\phi_E+\delta \phi_F)$, which coincides with $\mathcal I(\phi_E)$ when $\delta$ is small enough. This shows $(i)$. 

The item $(ii)$ can be proved along the same lines as Lemma~\ref{prelemme}, using the fact that $E+F$ is snc and transverse to the fibers.

As for item $(iii)$, we have pointwise convergence $c(\phi_\delta)\to c(\phi_L)$ on a Zariski open set of each $X_t$, $t\in D_3$, cf. Definition~\ref{geod curv}. Moreover, the Kähler metric $(dd^c \phi_\delta)_{\rm ac}$ on $\cX$ is uniformly bounded above by a fixed Kähler metric on $\cX$ (for instance, $\om_L+dd^c \phi_A$). In particular $c(\phi_\delta)$ is uniformly bounded above (say over compact subsets of $D_3$) and one can apply Lebesgue dominated convergence theorem to conclude. 
\end{proof}

The proof of Theorem~\ref{thmB} is now complete.
\end{proof}

\begin{remark} 
\label{rem Dstar}
The following limit argument shows that we can take $D^\star\subset D$ to be the set of $t\in D$ such that the following hold:
\smallskip
  
\noindent $\bullet$ the metric $h_\cF$ is smooth locally near $t$;

\noindent $\bullet$ the fiber $\cF_t$ coincides with $\displaystyle H^0\left(X_t, (K_{X_t}+ L)\otimes \mathcal I(h_L|_{X_t})\right)$.  
\medskip

\noindent Let $0\in D$ be a point which satisfies these requirements. Let $U\Subset \cX \setminus (h_L= \infty)$ be any open subset of $\cX$
whose closure does not meet the singular locus of the metric $h_L$.
Then we have 
\begin{equation}\label{cont}
\int_{U \cap X_0}c (\phi_L) u\wedge \overline{u} e^{-\phi_L} = \lim_{t\to 0}
\int_{U \cap X_{t}} c (\phi_L) u\wedge \overline{u} e^{-\phi_L}
\end{equation}
since all the objects involved are non-singular.

The next observation is that since $h_\cF$ is smooth near $0$ -by assumption-,
the function $t\mapsto \langle\Theta_{h_\cF}(\cF)u, u\rangle_t$ is thus continuous at 0. 
Theorem \ref{thmB} combined with \eqref{cont} and the positivity of $c(\phi)$ shows that we have 
$$\int_{U \cap X_0}c (\phi_L) u\wedge \overline{u} e^{-\phi_L} \leq \langle\Theta_{h_\cF}(\cF)u, u\rangle (0).$$
It follows that the estimate in \eqref{c1} of Theorem \ref{thmB} extends across $0\in D$ as well.
\end{remark}

\bibliographystyle{smfalpha}
\bibliography{biblioBob}

\end{document}